\pdfoutput=1 

\documentclass[11pt]{amsart}
\usepackage[utf8]{inputenc}
\usepackage{amssymb,amsmath,amsthm,enumerate,enumitem,colonequals,mlmodern,tikz-cd,microtype,dirtytalk}
\usepackage[mathcal]{eucal}
\usepackage[top=3.75cm, bottom=3cm, left=3.5cm, right=3.5cm]{geometry}

\makeatletter
\@namedef{subjclassname@2020}{\textup{2020} Mathematics Subject Classification}
\makeatother

\usepackage{hyperref}
\hypersetup{
    colorlinks=true,
    linkcolor= blue,
    urlcolor= blue,
    citecolor= green,
}

\def\eqref#1{\textcolor{blue}{(\ref{#1})}}

\PassOptionsToPackage{hyphens}{url}
\usepackage{hyperref}
\hypersetup{bookmarksdepth=2}

\usepackage[nameinlink]{cleveref} 
\Crefformat{section}{#2\S#1#3} 
\Crefmultiformat{section}{#2\S\S#1#3}{ and~#2#1#3}{, #2#1#3}{, and~#2#1#3}
\Crefname{hypothesis}{hypothesis}{hypotheses}

\Crefname{equation}{}{}
\Crefname{item}{}{}

\newcommand{\Hom}[4][]{\operatorname{Hom}^{#1}_{#2}(#3,#4)}
\newcommand{\HomT}[3][]{\operatorname{Hom}^{#1}_{\mathsf{T}}(#2,#3)}

\newtheorem{theoremintro}{Theorem}

\theoremstyle{definition}
\newtheorem{definitionintro}{Definition}

\newtheorem{exampleintro}{Example}

\newcommand{\hocolim}{\operatornamewithlimits{\underset{\boldsymbol{\xrightarrow{\hspace{0.85cm}}}}{{Hocolim}}}}
\makeatletter
\newcommand{\colim@}[2]{%
  \vtop{\m@th\ialign{##\cr
    \hfil$#1\operator@font colim$\hfil\cr
    \noalign{\nointerlineskip\kern1.5\ex@}#2\cr
    \noalign{\nointerlineskip\kern-\ex@}\cr}}%
}
\newcommand{\colim}{%
  \mathop{\mathpalette\colim@{\rightarrowfill@\textstyle}}\nmlimits@
}
\makeatother

\usepackage[pagewise]{lineno}
\let\oldequation\equation
\let\oldendequation\endequation
\renewenvironment{equation}{\linenomathNonumbers\oldequation}{\oldendequation\endlinenomath}
\expandafter\let\expandafter\oldequationstar\csname equation*\endcsname
\expandafter\let\expandafter\oldendequationstar\csname endequation*\endcsname
\renewenvironment{equation*}{\linenomathNonumbers\oldequationstar}{\oldendequationstar\endlinenomath}
\let\oldalign\align
\let\oldendalign\endalign
\renewenvironment{align}{\linenomathNonumbers\oldalign}{\oldendalign\endlinenomath}
\expandafter\let\expandafter\oldalignstar\csname align*\endcsname
\expandafter\let\expandafter\oldendalignstar\csname endalign*\endcsname
\renewenvironment{align*}{\linenomathNonumbers\oldalignstar}{\oldendalignstar\endlinenomath}


\newcounter{intro}
\newcounter{HypCounter}

\theoremstyle{plain}
\newtheorem{theorem}{Theorem}[section]
\newtheorem{lemma}[theorem]{Lemma}
\newtheorem{corollary}[theorem]{Corollary}
\newtheorem{proposition}[theorem]{Proposition}

\theoremstyle{definition}

\newtheorem{convention}[theorem]{Convention}
\newtheorem{definition}[theorem]{Definition}
\newtheorem{example}[theorem]{Example}

\newtheorem*{hypothesis*}{Hypothesis}
\newtheorem{notation}[theorem]{Notation}

\newtheorem{remark}[theorem]{Remark}

\numberwithin{equation}{section}
\numberwithin{theorem}{section}

\title{Admissible subcategories and metric techniques}

\author[K. ~Manali Rahul]{Kabeer Manali Rahul}
\address{K. ~Manali Rahul,
Center for Mathematics and its Applications, 
Mathematical Science Institute, Building 145, 
The Australian National University, 
Canberra, ACT 2601, Australia}
\email{kabeer.manalirahul@anu.edu.au}

\date{\today}

\keywords{Semiorthogonal decomposition, metric techniques, approximable triangulated categories, admissible subcategories}

\subjclass[2020]{14A30 (primary), 18G80, 14A22}

\begin{document}
    
\begin{abstract}
    In this work, we provide a way of constructing new semiorthogonal decompositions using metric techniques (\`{a} la Neeman). Given a semiorthogonal decomposition on a category with a special kind of metric, which we call a compressible metric, we can construct new semiorthogonal decomposition on a category constructed from the given one using the aforementioned metric. In the algebro-geometric setting, this gives us a way of producing new semiorthogonal decompositions on  various small triangulated categories associated to a scheme, if we are given one. In the general setting, the work is related to that of Sun-Zhang, while its applications to algebraic geometry are related to the work of Bondarko and Kuznetsov-Shinder.
\end{abstract}

\maketitle

\section{Introduction}
Triangulated categories appear in diverse areas of mathematics, from geometry, to topology, to representation theory. Hence, there has been a lot of work on trying to understand the structure of triangulated categories. One of the most important tools for the same are semiorthogonal decompositions. These allow us to understand a complicated triangulated category by breaking it up into simpler pieces. These have found a lot of applications, especially in algebraic geometry and representation theory, including applications on the connections between these areas. With this in mind, it becomes an interesting problem to discover new semiorthogonal decompositions.

There has been recent work in constructing new semiorthogonal decompositions from a given one with applications to algebraic geometry. The results are of the following form; we are given a semiorthogonal decomposition on a triangulated category $\mathsf{S}_1$, and we are able to produce a semiorthogonal decomposition on a different triangulated category $\mathsf{S}_2$ which is related to $\mathsf{S}_1$ in some manner. Two important recent papers which prove results of this form are \cite{Kuznetsov/Shinder:2023} and \cite{Bondarko:2023}. The main results of this work are also of this form. We will compare the results with the ones existing in the literature at the end of this section.

To give the reader a more concrete idea about what these statements look like, let us state some applications of the main abstract results proved in this paper to the context of algebraic geometry. We give a consolidated result for schemes, algebras, and stacks. The conditions imposed in the theorem below are more stringent than required in some cases for ease of exposition. See the corresponding results in the main body of this work for the more general statements.

\begin{theoremintro}\label{Introduction Theorem A}
    Let $X$ be a noetherian, separated, finite-dimensional scheme, and $\mathcal{A}$ a coherent $\mathcal{O}_X$-algebra. Let $\mathcal{X}$ be a concentrated noetherian stack (concentrated means that the canonical map $\mathcal{X} \to \operatorname{Spec}(\mathbb{Z})$ is concentrated, see \cite[Definition 2.4]{Hall/Rydh:2017}) such that one of the conditions below hold,
    \begin{itemize}
        \item $\mathcal{X}$ has quasi-finite and separated diagonal.
        \item $\mathcal{X}$ is a Deligne-Mumford stack of characteristic zero.
    \end{itemize}
    
    Then, for any semiorthogonal decomposition $\langle \mathsf{A}, \mathsf{B\rangle}$ on $\mathsf{S}_1$, we get a semiorthogonal decomposition $\langle\operatorname{Coprod}(\mathsf{A}) \cap \mathsf{S}_2,\operatorname{Coprod}(\mathsf{B}) \cap \mathsf{S}_2\rangle $ on $\mathsf{S}_2$ for any $(\mathsf{S}_1,\mathsf{S}_2)$ as in the table below. Note that  $\operatorname{Coprod}(\mathsf{A})$ denotes the localising subcategory generated by $\mathsf{A}$ inside the derived category of quasicoherent sheaves (over the scheme/algebra/stack in question).
    \begin{table}[h]
    \centering
    \renewcommand{\arraystretch}{1.5}
    \begin{tabular}{||c| c | c |  c||} 
        \hline
        \textbf{Extra conditions} & \textbf{$\mathsf{S}_1$} & \textbf{$\mathsf{S}_2$} & \textbf{Reference}\\
        \hline\hline
        -- &  $\mathbf{D}^{\operatorname{perf}}(X)$ & $\mathbf{D}^{-}_{\operatorname{coh}}(X)$ & \Cref{Remark Main result for Perf of schemes}\\
        \hline
        $\mathsf{B}$ is admissible & $\mathbf{D}^{\operatorname{perf}}(X)$ & $\mathbf{D}^{b}_{\operatorname{coh}}(X)$ & \Cref{Remark Main result for Perf of schemes}\\
        \hline
        $X$ is J-2 & $\mathbf{D}^{b}_{\operatorname{coh}}(X)$ & $\mathbf{D}^{+}_{\operatorname{coh}}(X)$ & \Cref{Theorem Main result for Dbcoh of schemes} \\
        \hline
        
        -- & $\mathbf{D}^{\operatorname{perf}}(\mathcal{A})$ & $\mathbf{D}^{-}_{\operatorname{coh}}(\mathcal{A})$ & \Cref{Theorem Main result for Perf of algebras}\\
        \hline
        $\mathsf{B}$ is admissible & $\mathbf{D}^{\operatorname{perf}}(\mathcal{A})$ & $\mathbf{D}^{b}_{\operatorname{coh}}(\mathcal{A})$ & \Cref{Theorem Main result for Perf of algebras}\\
        \hline
        $X$ is J-2 & $\mathbf{D}^{b}_{\operatorname{coh}}(\mathcal{A})$ & $\mathbf{D}^{+}_{\operatorname{coh}}(\mathcal{A})$ & \Cref{Theorem Main result for Dbcoh of algebras} \\
        \hline
        $X$ is J-2 \& $\mathsf{B}$ is admissible & $\mathbf{D}^{b}_{\operatorname{coh}}(\mathcal{A})$ & $\mathbf{D}^{b}_{\operatorname{coh}}(\operatorname{Inj}\mathcal{A})$ & \Cref{Theorem Main result for Dbcoh of algebras}\\
        \hline
        -- & $\mathbf{D}^{\operatorname{perf}}(\mathcal{X})$ & $\mathbf{D}^{-}_{\operatorname{coh}}(\mathcal{X})$ & \Cref{Theorem Main result for Perf of Stacks}\\
        \hline
        $\mathsf{B}$ is admissible & $\mathbf{D}^{\operatorname{perf}}(\mathcal{X})$ & $\mathbf{D}^{b}_{\operatorname{coh}}(\mathcal{X})$ & \Cref{Theorem Main result for Perf of Stacks}\\
        \hline
    \end{tabular}
    \end{table}
     
\end{theoremintro}

This theorem follows from more general results proved in \Cref{Section Main Results}. We now get into the details of those results, but again, not in their full generality.

We begin with the following definition. Note that some of the notation below is only used in the introduction.

\begin{definitionintro}\label{Introduction Definition A}
    Let $\mathsf{T}$ be a compactly generated triangulated category with a single compact generator $G$. Let $\big({}^\mathrm{I}\mathsf{T}_G^{\leq 0}, {}^\mathrm{I}\mathsf{T}_G^{\geq 0}\big)$ and $\big({}^\mathrm{II}\mathsf{T}_G^{\geq 0}, {}^\mathrm{II}\mathsf{T}_G^{\leq 0}\big)$ denote the standard t-structure and co-t-structure compactly generated by G respectively, see \Cref{Definition t-structure} and \Cref{Definition co-t-structure/weight structure}. We define the following full subcategories of $\mathsf{T}$,
    \begin{itemize}
        \item $\mathsf{T}^-_c \colonequals \bigcap_{n \in \mathbb{Z}} \mathsf{T}^c \star {}^\mathrm{I}\mathsf{T}_G^{\leq -n}$, that is, $\mathsf{T}^-_c$ is the full subcategory of $\mathsf{T}$ consisting of all objects F such that for all  $n\in \mathbb{Z}$, there exists a triangle $E_n \to F \to D_n \to \Sigma E_n$ for a compact object $E_n$ and with $D_n \in {}^\mathrm{I}\mathsf{T}_G^{\leq -n}$.
        \item ${}^\mathrm{I}\mathsf{T}^{b}_c = \{F \in \mathsf{T}^-_c : \HomT{\Sigma^{i} G}{F} = 0 
        $\text{ for }$ i>>0\}$
        \item $\mathsf{T}^{+}_c \colonequals \bigcap_{n \in \mathbb{Z}} \mathsf{T}^c \star {}^\mathrm{II}\mathsf{T}_G^{\geq n}$, that is, $\mathsf{T}^{+}_c$ is the full subcategory of $\mathsf{T}$ consisting of all objects F such that for all  $n\in \mathbb{Z}$, there exists a triangle $E_n \to F \to D_n \to \Sigma E_n$ for a compact object $E_n$ and with $D_n \in {}^\mathrm{II}\mathsf{T}_G^{\geq n}$.
        \item ${}^\mathrm{II}\mathsf{T}^{b}_c = \{F \in \mathsf{T}^+_c : \HomT{\Sigma^i G}{F} = 0 $\text{ for }$ i<<0\}$
    \end{itemize}
\end{definitionintro}

With this in mind, we introduce the following hypotheses.

\begin{hypothesis*}
    Let $\mathsf{T}$ be a compactly generated triangulated category with a single compact generator $G$. Then,
    \begin{enumerate}[label=(\Roman*)]
        \item\label{Hypothesis I} $\mathsf{T}$ satisfies Hypothesis (\textrm{I}) if $\HomT{\Sigma^{-i} G}{G} = 0$ for $i >> 0$.
        \item\label{Hypothesis II} $\mathsf{T}$ satisfies Hypothesis (\textrm{II}) if $\HomT{\Sigma^{i} G}{G} = 0$ for $i >> 0$.
    \end{enumerate}
\end{hypothesis*}

We note here that there are plenty of examples where these hypotheses hold.

\begin{exampleintro}\label{Introduction Example A}
    Hypothesis \ref{Hypothesis I} holds for the following categories,
    \begin{itemize}
        \item $\mathbf{D}_{\operatorname{Qcoh}}(X)$ for a noetherian scheme $X$.
        \item $\mathbf{D}_{\operatorname{Qcoh}}(\mathcal{A})$ for a quasicoherent $\mathcal{O}_X$-algebra over a noetherian scheme $X$.
        \item $\mathbf{D}_{\operatorname{Qcoh}}(\mathcal{X})$ for a 1-Thomason stack, see \cite[Definition 8.1]{Hall/Rydh:2017}.
    \end{itemize}

    Hypothesis \ref{Hypothesis II} holds for the following categories,
    \begin{itemize}
        \item $\mathbf{K}(\operatorname{Inj}X)$ for a noetherian J-2 scheme $X$.
        \item $\mathbf{K}_m(\operatorname{Proj}X)$ for a noetherian, separated, J-2 scheme $X$, where $\mathbf{K}_m(\operatorname{Proj}X)$ denote the mock homotopy category of projectives as defined in \cite{Murfet:2008}.
        \item $\mathbf{K}(\operatorname{Inj}\mathcal{A})$ for a quasicoherent $\mathcal{O}_X$-algebra over a noetherian J-2 scheme $X$.
    \end{itemize}
\end{exampleintro}

With these hypotheses, we have the following general results on producing semiorthogonal decompositions.

\begin{theoremintro}[\Cref{Corollary admissible categories from Tc to T-c} \& \Cref{Corollary admissible subcategories from Tc to Tbc with pre-approximability}]\label{Introduction Theorem B}
    Let $\mathsf{T}$ satisfy Hypothesis \ref{Hypothesis I}. If $\langle \mathsf{A}, \mathsf{B} \rangle$ is a semiorthogonal decomposition on $\mathsf{T}^c$, then $\langle \operatorname{Coprod}(\mathsf{A}) \cap \mathsf{T}^{-}_c, \operatorname{Coprod}(\mathsf{B}) \cap \mathsf{T}^{-}_c\rangle$ is a semiorthogonal decomposition on $\mathsf{T}^{-}_c$.

    If $\mathsf{B}$ is further assumed to be admissible, then $\langle \operatorname{Coprod}(\mathsf{A}) \cap {}^\mathrm{I}\mathsf{T}^{b}_c, \operatorname{Coprod}(\mathsf{B}) \cap {}^\mathrm{I}\mathsf{T}^{b}_c\rangle$ is a semiorthogonal decomposition on ${}^\mathrm{I}\mathsf{T}^{b}_c$.
\end{theoremintro}

\begin{theoremintro}[\Cref{Corollary admissible categories from Tc to T+c} \& \Cref{Corollary admissible subcategories from Tc to Tbc with co-approximability}]\label{Introduction Theorem C}
    Let $\mathsf{T}$ satisfy Hypothesis \ref{Hypothesis II}. If $\langle \mathsf{A}, \mathsf{B} \rangle$ is a semiorthogonal decomposition on $\mathsf{T}^c$, then $\langle \operatorname{Coprod}(\mathsf{A}) \cap \mathsf{T}^{+}_c, \operatorname{Coprod}(\mathsf{B}) \cap \mathsf{T}^{+}_c\rangle$ is a semiorthogonal decomposition on $\mathsf{T}^{+}_c$.

    If $\mathsf{B}$ is further assumed to be admissible, then $\langle \operatorname{Coprod}(\mathsf{A}) \cap {}^\mathrm{II}\mathsf{T}^{b}_c, \operatorname{Coprod}(\mathsf{B}) \cap {}^\mathrm{II}\mathsf{T}^{b}_c\rangle$ is a semiorthogonal decomposition on ${}^\mathrm{II}\mathsf{T}^{b}_c$
    .
\end{theoremintro}

Then, note that to get \Cref{Introduction Theorem A} from \Cref{Introduction Theorem B} and \Cref{Introduction Theorem C}, we need to compute the subcategories defined in \Cref{Introduction Definition A}. It turns out that some of the work has already been done for us. For the examples related to Hypothesis \ref{Hypothesis I}, the computation of $\mathsf{T}^{-}_c$ and ${}^\mathrm{I}\mathsf{T}^{b}_c$ follows from \cite{Neeman:2022, DeDeyn/Lank/ManaliRahul:2024a, Hall/Lamarche/Lank/Peng:2025, DeDeyn/Lank/ManaliRahul/Peng:2025}.

For the examples related to Hypothesis \ref{Hypothesis II}, more work is required to compute $\mathsf{T}^{+}_c$ and ${}^\mathrm{II}\mathsf{T}^{b}_c$. The main issue is that a priori there is not much control on the co-t-structure compactly generated by a compact object which is used to define these categories, see \Cref{Introduction Definition A}. Let us take the case of $\mathsf{T} = \mathbf{K}(\operatorname{Inj} X)$ for a noetherian J-2 scheme $X$ to illustrates our main strategy to solve this issue. $\mathbf{K}(\operatorname{Inj} X)$ has a standard co-t-structure, given by the brutal truncation. This co-t-structure is much easier to work with, because of the simple form of the truncations. Hence, what we need is to connect these two co-t-structures. One way to do this is through the theory of co-approximable triangulated categories, as introduced in \cite{ManaliRahul:2025} which is a variation of the notion of approximable triangulated categories introduced by Neeman, see \cite{Neeman:2021a, Neeman:2021b, Neeman:2021c, Neeman:2022}. For more details on the connection, we refer the reader to the \Cref{Lemma closure of compacts for Co-approx} and \Cref{Theorem J-2 implies weak co-approximablility for schemes}.

For the reader familiar with the theory of approximable triangulated categories, we note here that we prove (weak) approximability like structures for the homotopy category of injectives associated to a large class of schemes, and algebras over schemes, see \Cref{Theorem J-2 implies weak co-approximablility for schemes} and \Cref{Theorem J-2 implies weak co-approximablility for algebras}. Similar result also for the mock homotopy category of projectives, see \Cref{Proposition Kmproj is Co-approx}. These results should be of independent interest with the growing theory surrounding approximable triangulated categories, and in general metric techniques for triangulated categories.

Let us finish by discussing the related results of \cite{Kuznetsov/Shinder:2023,Bondarko:2023}. Both of these works have some abstract results, but for the sake of brevity, we compare the applications to algebraic geometry which are related to \Cref{Introduction Theorem A}, see \cite[Theorem 3.2.7]{Bondarko:2023} and \cite[Corollary 6.5]{Kuznetsov/Shinder:2023}. In both these works, the assumptions on the scheme is stronger; both assume at least properness over a noetherian ring. But, the results proved are also stronger, in the sense that the need for the extra hypothesis on being admissible is not required. 

The overarching theme seems to be that either one assumes these extra conditions on the scheme, or one would need to assume the existence of an extra adjoint. In the first approach, which is taken by both \cite{Kuznetsov/Shinder:2023,Bondarko:2023}, the idea is to use the powerful representability theorems proved in \cite{Bondal/VanDenBergh:2003,Rouquier:2003,Neeman:2021b}. 
 
In this work, we take the second approach, which does not require these powerful representability theorems. Another recent paper which works on a similar problem in a very general setting is \cite{Sun/Zhang:2021}, see \Cref{Remark on paper by Sun and Zhang} for a discussion.

\subsection*{Acknowledgments} This work is part of the author's PhD work at the Australian National University under the supervision of Amnon Neeman. The author would like to thank him for
generously sharing his knowledge and his help throughout the PhD. The author would also like to thank Pat Lank and Timothy De Deyn for many useful discussions, especially regarding stacks and noncommutative algebras over schemes. Further, the author thanks BIREP group at Universit\"{a}t Bielefeld and the Algebraic Geometry group at Universit\`{a} degli Studi di Milano for their hospitality during the author's stay there.

Finally, the author was supported by an Australian Government Research Training Program scholarship, and was supported under the ERC Advanced
Grant 101095900-TriCatApp, the Deutsche Forschungsgemeinschaft (SFB-TRR 358/1 2023 - 491392403), and the Australian Research Council Grant DP200102537 at different points during the PhD.

\section{Background and notation}
Throughout this work, we will use $\Sigma$ to denote the shift functor in the definition of triangulated categories. For any functor $F : \mathsf{S} \to \mathsf{T}$, $F(\mathsf{S})$ will denote the essential image of $\mathsf{S}$ under $F$. Further, all the gradings involved in this work will be cohomological. For full subcategories $\mathsf{A}$ and $\mathsf{C}$ of a triangulated category $\mathsf{T}$, $\mathsf{A} \star \mathsf{C}$ will denote the full subcategory consisting of objects $B$ such that there exists a triangle $A \to B \to C \to \Sigma A$ with $A \in \mathsf{A}$ and $C \in \mathsf{C}$. For any full subcategory $\mathsf{A}$ of $\mathsf{T}$, $\mathsf{A}^{\perp}$ (resp.\ ${}^{\perp}\mathsf{A}$) will denote the full subcategory of all objects $B$ such that $\HomT{A}{B} = 0$ (resp.\ $\HomT{B}{A} = 0$) for all $A \in \mathsf{A}$. Finally, for a triangulated category $\mathsf{T}$ with coproducts, we denote the full subcategory of compact objects by $\mathsf{T}^c$.

We now recall some background material related to metrics on triangulated categories.

\subsection*{Metrics, generating sequences, and weak co-approximability}
We begin with some notation from \cite{Bondal/VanDenBergh:2003, Neeman:2021b}.

\begin{definition}\label{Notation from Neeman}
    Let $\mathsf{C}$ be a full subcategory of a triangulated category $\mathsf{T}$. Then,
    \begin{itemize}
        \item $\operatorname{smd}(\mathsf{C})$ (resp.\ $\operatorname{coprod}(\mathsf{C})$) denotes the closure of $\mathsf{C}$ under summands (resp.\ finite coproducts and extensions).
        \item If $\mathsf{T}$ has coproducts, then $\operatorname{Coprod}(\mathsf{C})$ denotes the closure of $\mathsf{C}$ under coproducts and extensions. Note that if $\mathsf{C} = \Sigma \mathsf{C}$, then $\operatorname{Coprod}(\mathsf{C})$ is the localising subcategory of $\mathsf{T}$ generated by $\mathsf{C}$.
        \item Let $G$ be an object of $\mathsf{T}$. Then, for any $a \leq b$, we define,
        \[\langle G \rangle^{[a,b]} \colonequals \operatorname{smd}(\operatorname{coprod}(\{\Sigma^{-i} G : a \leq i \leq b\}))\]
        We also extend this definition for $a=-\infty$ and $b=\infty$ in the obvious way. 
        \item We define $\langle G \rangle \colonequals \bigcup_{n \geq 0} \langle G \rangle^{[-n,n]}$. This is in fact the thick subcategory generated by $G$. If $G$ is a compact generator for a triangulated category $\mathsf{S}$, then the subcategory of compacts $\mathsf{S}^c = \langle G \rangle$.
        \item Let $G$ be an object of $\mathsf{T}$. If $\mathsf{T}$ has coproducts, we define,
        \[\overline{\langle G \rangle}^{[a,b]} = \operatorname{smd}(\operatorname{Coprod}(\{\Sigma^{-i} G : a \leq i \leq b\}))\]
        We also extend this definition for $a=-\infty$ and $b=\infty$ in the obvious way.
    \end{itemize}
    
\end{definition}

We begin by stating the definition of an extended good metric for a triangulated category, which is just the $\mathbb{Z}$-graded version of Neeman's notion of a good metric, see \cite[Definition 10]{Neeman:2020} and \cite[Definition 3.1]{ManaliRahul:2025}.

\begin{definition}\label{Definition Metric}
    Let $\mathsf{T}$ be a triangulated category. An \emph{extended good metric
    } is a decreasing sequence of strictly full subcategories $\{\mathcal{M}_n\}_{n \in \mathbb{Z}}$ each closed under extensions and containing $0$ such that,
    \[\Sigma^{-1} \mathcal{M}_{n+1} \cup \mathcal{M}_{n+1} \Sigma\mathcal{M}_{n+1} \subseteq \mathcal{M}_n\]
    for all integers $n$. From now on, when we say metric, we would always mean an extended good metric.
    
    We say an extended good metric $\{\mathcal{R}_n\}_{n \in \mathbb{Z}}$ is an \emph{orthogonal metric} if ${}^{\perp}[(\mathcal{R}_n)^{\perp}] = \mathcal{R}_n$ for all integers $n$. Finally, we say two metrics $\mathcal{M}$ and $\mathcal{N}$ are \emph{equivalent} if for all $i$, there exists $j \geq 0$ such that $\mathcal{M}_{i+j} \subseteq \mathcal{N}_j \subseteq \mathcal{M}_{i-j}$, and we say that they are \emph{$\mathbb{N}$-equivalent} if the $j$ can be chosen independent of $i$.
\end{definition}

We now recall a few notions from \cite{ManaliRahul:2025}.
\begin{definition}\label{Definition generating sequence}
    Let $\mathsf{T}$ be a compactly generated triangulated category. Then,
    \begin{itemize}
        \item A \emph{pre-generating sequence} is a sequence $\{\mathcal{G}_n\}_{n \in \mathbb{Z}}$ of full subcategories of compact objects such that $\operatorname{smd}(\operatorname{coprod}(\bigcup_{n \in \mathbb{Z}}\mathcal{G}_{n}))) = \mathsf{T}^c$. Given a pre-generating sequence $\mathcal{G}$, and $a \leq b$, we define the full subcategory $\mathcal{G}^{[a,b]} = \operatorname{smd}(\operatorname{coprod}(\bigcup_{i=a}^b)\mathcal{G}^{i})$. We also extend this definition for $a=-\infty$ and $b=\infty$ in the obvious way.
        \item A pre-generating sequence $\mathcal{G}$ is a \emph{generating sequence} if \[\Sigma^{-1}\mathcal{G}^n \cup \mathcal{G}^n \cup \Sigma\mathcal{G}^n \subseteq \mathcal{G}^{[n-1,n+1]}\]
        for all $n \in \mathbb{Z}$. It is further a \emph{finite generating sequence} if $\mathcal{G}^n$ consists of finitely many objects for each integer $n$.
        \item Given a generating sequence $\mathcal{G}$, we can define an extended good metric $\mathcal{M}^{\mathcal{G}}$ on $\mathsf{T}^c$ and an orthogonal metric $\mathcal{R}^{\mathcal{G}}$ on $\mathsf{T}$ defined by,
        \[\mathcal{M}^{\mathcal{G}}_n = \mathcal{G}^{(-\infty, -n]} =  \operatorname{smd}\Big(\operatorname{coprod}\Big(\bigcup_{i \leq -n}\mathcal{G}^{i}\Big)\Big), \ \ \mathcal{R}^{\mathcal{G}}_n = {}^{\perp}[(\mathcal{M}^{\mathcal{G}}_n)^{\perp}] \]
        for all $n \in \mathbb{Z}$.
    \end{itemize}
\end{definition}

\begin{example}\label{Example generating sequence approx/co-approx}
    Let $\mathsf{T}$ be a compactly generated triangulated category with a single compact generator $G$. In this case we can define the following two finite generating sequences (\Cref{Definition generating sequence}),
    \begin{itemize}
        \item $\mathcal{G}^n = \{\Sigma^{-n} G\}$. In this case $\mathcal{M}^{\mathcal{G}}_n = \langle G \rangle^{(-\infty,-n]}$, see \Cref{Notation from Neeman}. Further, $\mathcal{R}^{\mathcal{G}} = \mathsf{T}_G^{\leq -n}$ where $(\mathsf{T}_G^{\leq 0}, \mathsf{T}_G^{\geq 0})$ is the t-structure generated by $G$, see \Cref{Definition t-structure}. Note that the metric $\mathcal{R}^{\mathcal{G}} \cap \mathsf{T}^c$ is equal to the metric $\mathcal{M}^{\mathcal{G}}$ by the remark at the end of \Cref{Definition t-structure}.
        \item $\mathcal{G}^n = \{\Sigma^n G\}$. In this case $\mathcal{M}^{\mathcal{G}}_n = \langle G \rangle^{[n,\infty)}$, see \Cref{Notation from Neeman}. Further, $\mathcal{R}^{\mathcal{G}} = \Sigma^{-n}\mathsf{U}_G$ where $(\mathsf{U}_G, \mathsf{V}_G)$ is the co-t-structure generated by $G$, see \Cref{Definition co-t-structure/weight structure}. Note that the metric $\mathcal{R}^{\mathcal{G}} \cap \mathsf{T}^c$ is equivalent to the metric $\mathcal{M}^{\mathcal{G}}$ by the remark at the end of \Cref{Definition co-t-structure/weight structure}.
    \end{itemize}
\end{example}

For the sake of completeness, we recall the definitions of a (compactly generated) t-structure and a (compactly generated) co-t-structure.
\begin{definition}\label{Definition t-structure}\cite[D\'{e}finition 1.3.1]{Beilinson/Bernstein/Deligne:1982}
    A \emph{t-structure} on a triangulated category $\mathsf{T}$ is a pair $(\mathsf{T}^{\leq 0}, \mathsf{T}^{\geq 0})$ of strictly full subcategories such that,
    \begin{itemize}
        \item $\Sigma\mathsf{T}^{\leq 0} \subseteq \mathsf{T}^{\leq 0}$, $\Sigma^{-1}\mathsf{T}^{\geq 0} \subseteq \mathsf{T}^{\geq 0}$, and $\Hom{\mathsf{T}}{\Sigma\mathsf{T}^{\leq 0}}{\mathsf{T}^{\geq 0}} = 0$. We define $\Sigma^{-n}\mathsf{T}^{\leq 0} = \mathsf{T}^{\leq n}$ and $\Sigma^{-n}\mathsf{T}^{\geq 0} = \mathsf{T}^{\geq n}$.
        \item For all $Y \in \mathsf{T}$, there exists a triangle $X \to Y \to Z \to \Sigma X$ with $X \in \Sigma\mathsf{T}^{\leq 0}$ and $Z \in \mathsf{T}^{\geq 0}$.
    \end{itemize}
    Given a t-structure, we get an orthogonal metric associated to it on $\mathsf{T}$ which is given by $\{\mathsf{T}^{\leq -n}\}_{n \in \mathbb{Z}}$. Further, if $\mathsf{T}$ has coproducts, we also get a metric associated to it on $\mathsf{T}^c$ which is given by $\{\mathsf{T}^{\leq -n} \cap \mathsf{T}^c\}_{n \in \mathbb{Z}}$.
    
    If $\mathsf{T}$ has coproducts, then for any compact object $G \in \mathsf{T}^c$, $(\mathsf{T}_G^{\leq 0},\mathsf{T}_G^{\geq 0})$ is a t-structure by \cite[Theorem A.1]{AlonsoTarrio/JeremiasLopez/SoutoSalorio:2003} where,
    \[\mathsf{T}_G^{\leq 0} = \operatorname{Coprod}(\{\Sigma^i 
    G: i \geq 0\}), \mathsf{T}_G^{\geq 0} = (\Sigma \mathsf{T}_G^{\leq 0})^{\perp}\]
    This is said to be the \emph{t-structure generated by $G$}. We note here that $\mathsf{T}_G^{\leq 0} \cap \mathsf{T}^c = \langle G \rangle^{(-\infty,0]}$, which implies that the metric associated to it on $\mathsf{T}^c$ is the same as the metric $\{\langle G \rangle^{(-\infty,-n]}\}_{n \in \mathbb{Z}}$.
\end{definition}

\begin{remark}\label{Remark aisle and coaisle}
    Let $\mathsf{T}$  be a triangulated category. 
    \begin{itemize}
        \item A strictly full subcategory $\mathsf{U}$ is called a \emph{aisle} if $\Sigma\mathsf{U} \subseteq \mathsf{U}$, and the inclusion $\mathsf{U} \to \mathsf{T}$ has a right adjoint.
        \item A strictly full subcategory $\mathsf{U}$ is called a \emph{coaisle} if $\Sigma^{-1}\mathsf{V} \subseteq \mathsf{V}$, and the inclusion $\mathsf{V} \to \mathsf{T}$ has a left adjoint.
    \end{itemize}
    Then, we have the following,
    \begin{itemize}
        \item If $(\mathsf{U}, \mathsf{V})$ is a t-structure on $\mathsf{T}$ then $\mathsf{U}$ is an aisle and $\mathsf{V}$ is a coaisle.
        \item If $\mathsf{U}$ is an aisle, then, $(\mathsf{U}, \Sigma \mathsf{U}^{\perp})$ is a t-structure on $\mathsf{T}$.
        \item If $\mathsf{V}$ is an coaisle, then, $\big(\Sigma^{-1}({}^{\perp}\mathsf{V}),\mathsf{V}\big)$ is a t-structure on $\mathsf{T}$.
    \end{itemize}
    
\end{remark}

\begin{definition}\label{Definition co-t-structure/weight structure}
    \cite[Definition 1.1.1]{Bondarko:2010} and \cite[Definition 1.4]{Pauksztello:2008} A \emph{co-t-structure} on a triangulated category $\mathsf{T}$ is a pair $(\mathsf{T}^{\geq 0}, \mathsf{T}^{\leq 0})$ of strictly full subcategories such that,
    \begin{itemize}
        \item $\Sigma\mathsf{T}^{\leq 0} \subseteq \mathsf{T}^{\leq 0}$, $\Sigma^{-1}\mathsf{T}^{\geq 0} \subseteq \mathsf{T}^{\geq 0}$, and $\Hom{\mathsf{T}}{\mathsf{T}^{\geq 0}}{\Sigma\mathsf{T}^{\leq 0}} = 0$. We define $\Sigma^{-n}\mathsf{T}^{\leq 0} = \mathsf{T}^{\leq n}$ and $\Sigma^{-n}\mathsf{T}^{\geq 0} = \mathsf{T}^{\geq n}$.
        \item For all $Y \in \mathsf{T}$, there exists a triangle $X \to Y \to Z \to \Sigma X$ with $Z \in \Sigma\mathsf{T}^{\leq 0}$ and $X \in \mathsf{T}^{\geq 0}$.
    \end{itemize}
    Given a co-t-structure, we get an orthogonal metric associated to it on $\mathsf{T}$ which is given by $\{\mathsf{T}^{\geq n}\}_{n \in \mathbb{Z}}$. Further, if $\mathsf{T}$ has coproducts, we also get a metric associated to it on $\mathsf{T}^c$ which is given by $\{\mathsf{T}^{\geq n} \cap \mathsf{T}^c\}_{n \in \mathbb{Z}}$.
    
    If $\mathsf{T}$ has coproducts, then for any compact object $G \in \mathsf{T}^c$, $(\mathsf{U}_G,\mathsf{V}_G)$ is a co-t-structure by \cite[Theorem 5]{Pauksztello:2012} where,
    \[\mathsf{V}_G = (\{\Sigma^i 
    G: i \leq 0\})^{\perp}, \mathsf{U}_G = {}^{\perp}(\Sigma \mathsf{V}_G)\]
    This is said to be the \emph{co-t-structure generated by $G$}. We note here that $\mathsf{U}_G^{\leq 0} \cap \mathsf{T}^c \subseteq \langle G \rangle ^{[-1,\infty)}$ by \cite[Theorem 2.3.4]{Bondarko:2022}, see \Cref{Notation from Neeman}. This implies that the metric associated to it on $\mathsf{T}^c$ is the $\mathbb{N}$-equivalent to the metric $\{\langle G \rangle^{[n,\infty)}\}_{n \in \mathbb{Z}}$, see \Cref{Definition Metric}.
\end{definition}

\begin{definition}\label{Definition preferred equivalence class}
    For a triangulated category $\mathsf{T}$, two t-structures (resp.\ co-t-structures) are said to be equivalent if the corresponding metric on $\mathsf{T}$ are equivalent, see \Cref{Definition Metric}, \Cref{Definition t-structure}, and \Cref{Definition co-t-structure/weight structure}. Note that this is the same as the metrics being $\mathbb{N}$-equivalent.

    Further, let $\mathsf{T}$ be a compactly generated triangulated category with a compact generator $G$. Then, we say a t-structure (resp.\ co-t-structure) $(\mathsf{U},\mathsf{V})$ lies in the 
    \begin{itemize}
        \item \emph{preferred quasiequivalence class of t-structures (resp.\ co-t-structures)} if the metric corresponding it to it on $\mathsf{T}^c$ is equivalent to the metric corresponding to the t-structure (resp.\ co-t-structure) generated by $G$, see \Cref{Definition t-structure} (resp.\ \Cref{Definition co-t-structure/weight structure}). Note that this is the same as the existence of a positive integer $n$ such that $\langle G \rangle^{(-\infty,-n]} \subseteq \mathsf{U} \cap \mathsf{T}^c \subseteq \langle G \rangle^{(-\infty,n]}$ (resp.\ $\langle G \rangle^{[n,\infty)} \subseteq \mathsf{U} \cap \mathsf{T}^c \subseteq \langle G \rangle^{[-n,\infty)}$), see \Cref{Notation from Neeman}.
        \item \emph{preferred equivalence class of t-structures (resp.\ co-t-structures)} if the metric corresponding it to it on $\mathsf{T}$ is equivalent to the orthogonal metric corresponding to the t-structure (resp.\ co-t-structure) generated by $G$, see \Cref{Definition t-structure} (resp.\ \Cref{Definition co-t-structure/weight structure}).
    \end{itemize}
\end{definition}
We now define the closure of compacts in the specific context we are interested in. For a more general definition, see \cite[Definition 3.17]{ManaliRahul:2025}.
\begin{definition}\label{Definition closure of compacts}
    Let $\mathsf{T}$ be a compactly generated triangulated category with a generating sequence $\mathcal{G}$, and with an orthogonal metric $\mathcal{R}$ on $\mathsf{T}$. Then, we define,
    \begin{itemize}
        \item $\overline{\mathsf{T}^c} \colonequals \bigcap_{n \in \mathbb{Z}} \mathsf{T}^c \star \mathcal{R}_n$, that is, $\overline{\mathsf{T}^c}$ is the full subcategory of $\mathsf{T}$ consisting of all objects F such that for all  $n\in \mathbb{Z}$, there exists a triangle $E_n \to F \to D_n \to \Sigma E_n$ with $E_n \in \mathsf{T}^c$ and $D_n \in \mathcal{R}_n$. We call $\overline{\mathsf{T}^c}$ the \emph{closure of the compacts}.
        \item $\mathcal{G}^{\perp} \colonequals \bigcup_{n \in \mathbb{Z}}(\mathcal{M}_n^{\mathcal{G}})^{\perp}$, see \Cref{Definition generating sequence} for the definition of $\mathcal{M}_n^{\mathcal{G}}$.
        \item $\mathsf{T}^b_c \colonequals \overline{\mathsf{T}^c} \cap \mathcal{G}^{\perp}$. We call $\mathsf{T}^b_c$ the \emph{bounded objects in the closure of compacts}.
    \end{itemize}
    Note that these notions only depends on the equivalence class of the metric $\mathcal{R}$.
    
    If we are given a triangulated category with a generating sequence, and we do not specify an orthogonal metric on $\mathsf{T}$, by default we will compute the closure of compact with respect to the orthogonal metric $\mathcal{R}^{\mathcal{G}}$, see \Cref{Definition generating sequence}. 
\end{definition}

The following lemma follows from \cite[Lemma 6.2]{ManaliRahul:2025}, but we give a short proof for the sake of completeness.

\begin{lemma}\label{Lemma approximating sequence for closure of compacts}
    Let $\mathsf{T}$ be a compactly generated triangulated category with a generating sequence $\mathcal{G}$ (\Cref{Definition generating sequence}), and with an orthogonal metric $\mathcal{R}$ (\Cref{Definition Metric}) on $\mathsf{T}$ such that for all $n$, $\HomT{\mathcal{G}^n}{\mathcal{R}_i} = 0$ for $i > >0$. Then, an object $F$ lies in the closure of compacts $\overline{\mathsf{T}^c}$, if and only if there exists a sequence $E_1 \to E_2 \to E_3 \to E_4 \to \cdots$ in $\mathsf{T}^c$ mapping to $F$ such that $\operatorname{Cone}(E_n \to F) \in \mathcal{R}_{n}$. 
    
    Furthermore, for any such sequence, $\hocolim E_i \to F$ is an isomorphism.
\end{lemma}
\begin{proof}
    If an object $F$ has such a sequence mapping to it, then by \Cref{Definition closure of compacts} it lies in the closure of compacts.
    
    We now prove the converse statement.
    First, note that as $\mathcal{G}$ is a generating sequence, $\mathsf{T}^c = \bigcup_{n \geq 0}\mathcal{G}^{[-n,n]}$, see \Cref{Definition generating sequence}. So, by the hypothesis, we get that for any $H \in \mathsf{T}^c$, $\HomT{H}{\mathcal{R}_i} = 0$ for $i > >0$. Using this, it is easy to show that the closure of compacts $\overline{\mathsf{T}^c}$ is triangulated, see \Cref{Definition closure of compacts}. Now, let $F \in \overline{\mathsf{T}^c}$. We will construct the required sequence inductively. The base case holds by \Cref{Definition closure of compacts}. So, we just need to show the inductive step. Suppose we have the sequence $E_1 \to \cdots \to E_n$ satisfying the required properties. If $D_n \colonequals \operatorname{Cone}(E_n \to F)$, then $D_n$ lies in $\mathcal{R}_n \cap \overline{\mathsf{T}^c}$. By the argument at the beginning of this proof, there exists an integer $n_1$, which we can choose so that $n_1 > n$, such that $\HomT{E_n}{\mathcal{R}_i} = 0$ for all $i \geq n_1$. As $D_n \in \overline{\mathsf{T}^c}$, there exists a triangle $\widetilde{E}_n \to D_n  \to D_{n+1} \to \Sigma \widetilde{E}_n$ with $\widetilde{E}_n \in \mathsf{T}^c$ and $D_{n+1} \in \mathcal{R}_{n_1}$. Applying the octahedral axiom to the composible morphisms $F \to D_n \to D_{n+1}$ gives us the required object $E_{n+1} \in \mathsf{T}^c$, and the triangle $E_{n+1} \to F \to D_{n+1} \to \Sigma E_{n+1}$.

    By the previous paragraph, we have that for any compact object $H$, $\Hom{\mathsf{T}}{H}{-}$ sends $E_n \to F$ to an isomorphism for large enough $n$.
    Let $E = \hocolim E_i$. We have a map $E \to F$ as the sequence maps to $F$. For any integer $i$, the natural map $\colim \Hom{\mathsf{T}}{\Sigma^i G}{E_i}{F} \xrightarrow{\sim} \Hom{\mathsf{T}}{\Sigma^i G}{E}{F}$ is an isomorphism by \cite[Lemma 2.8]{Neeman:1996}. So, combined with the observation at the beginning of this paragraph, we have that for any $H \in \mathsf{T}^c$, $\Hom{\mathsf{T}}{H}{-}$ sends $E \to F$ to an isomorphism. As $\mathsf{T}^c$ is compactly generated, this implies that the map $E \to F$ is an isomorphism. 
    %\Kabeer{Complete this proof}.
\end{proof}

We will have special notation for the subcategories defined in \Cref{Definition closure of compacts} in some cases, partly to be consistent with the notation in the literature, and partly to highlight these cases as these will be the most important examples for the purpose of applications. 

\begin{convention}\label{Convention closure of compacts for approx/co-approx}
    Let $\mathsf{T}$ be a compactly generated triangulated category with a single compact generator $G$.
    \begin{itemize}
        \item Suppose we equip $\mathsf{T}$ with the generating sequence given by $\mathcal{G}^n = \{\Sigma^{-n} G\}$ and with the orthogonal metric corresponding to some t-structure on $\mathsf{T}$, see \Cref{Definition t-structure}. Then, the closure of compacts is denote by $\mathsf{T}^-_c$ and $\mathcal{G}^{\perp}$ is denoted by $\mathsf{T}^+$.
        \item Suppose we equip $\mathsf{T}$ with the generating sequence given by $\mathcal{G}^n = \{\Sigma^{n} G\}$ and with the orthogonal metric corresponding to some co-t-structure on $\mathsf{T}$, see \Cref{Definition co-t-structure/weight structure}. Then, the closure of compacts is denote by $\mathsf{T}^-_c$ and $\mathcal{G}^{\perp}$ is denoted by $\mathsf{T}^-$.
    \end{itemize}
    We will mostly be working with t-structures (resp.\ co-t-structures) lying in the preferred quasiequivalence class, see \Cref{Definition preferred equivalence class}. Note that by \Cref{Lemma approximating sequence for closure of compacts}, the closure of compacts will be same if computed with respect to any metric coming from a t-strcucture (resp.\ co-t-structure) in the preferred quasiequivalence class.
\end{convention}

Finally, we state the definition of weak co-approximbaility, see \cite[Definitions 7.8 \& 7.9]{ManaliRahul:2025}

\begin{definition}\label{Definition Co-approx}
    Let $\mathsf{T}$ be a compactly generated triangulated category with a single compact generator $G$ and a co-t-structure $(\mathsf{U},\mathsf{V})$ such that $\Hom{\mathsf{T}}{G}{\Sigma^{-n}G} = 0$ and $\Sigma^{-n}G \in \mathsf{U}$ for $n>>0$. Then, 
    \begin{itemize}
        \item We say $\mathsf{T}$ is \emph{weakly co-approximable} if there exists $N \geq 0$ such that for all $F \in \mathsf{U}$, there exists a triangle $E \to F \to D \to \Sigma F$ with $D \in \Sigma^{-1}\mathsf{U}$ and $E \in \overline{\langle G \rangle}^{[-N,N]}$, see \Cref{Notation from Neeman}.
        \item We say $\mathsf{T}$ is \emph{weakly co-quasiapproximable} if there exists $N \geq 0$ such that for all $F \in \mathsf{U} \cap \mathsf{T}^+_c$, there exists a triangle $E \to F \to D \to \Sigma F$ with $D \in \Sigma^{-1}\mathsf{U} \cap \mathsf{T}^+_c$ and $E \in \overline{\langle G \rangle}^{[-N,N]}$, where we define $\mathsf{T}^+_c$ with respect to the given co-t-structure, see \Cref{Convention closure of compacts for approx/co-approx}.
    \end{itemize}
      
\end{definition}

\begin{convention}
    For the sake of making statements less cumbersome, from now onward we will sometimes abbreviate weakly co-approximable (resp.\ weak co-approximability) to weakly co-approx (resp.\ weak co-approx). Similarly, we will abbreviate weakly co-quasiapproximable (resp.\ weak co-approximability) to weakly co-quasiapprox (resp.\ weak co-quasiapprox).
\end{convention}

\begin{remark}\label{Remark Co-approx arbitrarily good approximation}
    Let $\mathsf{T}$ be weakly co-approx (resp.\ weakly co-quasiapprox) with a compact generator $G$, co-t-structure $(\mathsf{U},\mathsf{V})$, and integer $N \geq 0$ satisfying the deginition of weak co-approx (resp.\ weak co-quasiapprox). Then, by a simple induction argument, we can get arbitrarily good approximations of objects in $\mathsf{U}$ (resp.\ $\mathsf{U} \cap \mathsf{T}^+_c$). That is,
    for all $F \in \mathsf{U}$ (resp.\ for all $F \in \mathsf{U} \cap \mathsf{T}^+_c$) and for all $n \geq 1$, there exists a triangle $E_n \to F \to D_n \to \Sigma E_n$ with $E_n \in \overline{\langle G \rangle}^{[-N,N+n-1]}$ and $D_n \in \Sigma^{-n-1} \mathsf{U}$ (resp.\ $D_n \in \Sigma^{-n-1} \mathsf{U} \cap \mathsf{T}^+_c$).
\end{remark}

Although the following result does follow from the work in \cite{ManaliRahul:2025}, we give a proof for the sake of completion.
\begin{lemma}\label{Lemma Coapprox implies preferred equivaence class}
    %\Kabeer{Weak co-quasiapprox/ weak co-approx implies equivalence of metrics.}
    Let $\mathsf{T}$ be a weakly co-approx (resp.\ co-quasiapprox) triangulated category with a co-t-structure $(\mathsf{U},\mathsf{V})$ as in the definition of weak co-approx (resp.\ co-quasiapprox). Then, $(\mathsf{U},\mathsf{V})$ lies in the preferred quasiequivalence class of co-t-structures on $\mathsf{T}$, see \Cref{Definition preferred equivalence class}.
\end{lemma}
\begin{proof}
    Let $F \in \mathsf{U} \cap \mathsf{T}^+_c$.
    We begin by inductively producing a sequence $E_1 \to E_2 \to E_3 \to E_4 \to \cdots$ mapping to $F$ such that for all integers $n$, $E_n \in \overline{\langle G \rangle}^{[-N, \infty)}$ and $\operatorname{Cone}(E_n \to F) \in \Sigma^{-n-1} \mathsf{U}$ (resp.\ $\operatorname{Cone}(E_n \to F) \in \Sigma^{-n-1} \mathsf{U} \cap \mathsf{T}^+_c$) for the integer $N$ as in the definition of weak co-approx (resp.\ co-quasiapprox). The base case holds by \Cref{Definition Co-approx}. So, we just need to show the inductive step. Suppose we have the sequence $E_1 \to \cdots \to E_n$ satisfying the required properties. If $D_n \colonequals \operatorname{Cone}(E_n \to F)$, then $\Sigma^{n+1}D_n$ lies in $ \mathsf{U}$ (resp.\ $ \mathsf{U} \cap \mathsf{T}^+_c$). So, by \Cref{Definition Co-approx}, we get a triangle $\widetilde{E}_n \to D_n \to D_{n+1} \to \widetilde{E}_n$ with $\widetilde{E}_n \in \overline{\langle G \rangle}^{[-N, \infty)}$ and $D_{n+1} \in \Sigma^{-n-2}\mathsf{U}$ (resp.\ $D_{n+1} \in \Sigma^{-n-2}\mathsf{U} \cap \mathsf{T}^+_c$). By applying the octahedral axiom to $F \to D_n \to D_{n+1}$, we get the required triangle $E_{n+1} \to F \to D_{n+1} \to \Sigma E_{n+1}$ and a map $E_n \to E_{n+1}$ factoring $E_n \to F$.

    By \Cref{Definition Co-approx}, for any integer $i$, $\Hom{\mathsf{T}}{\Sigma^i{G}}{\Sigma^{-n}\mathsf{U}} = 0$ for $n>>0$. So, for any integer $i$, $\Hom{\mathsf{T}}{\Sigma^i G}{-}$ sends $E_j \to F$ to an isomorphism for large enough $j$.
    Let $E = \hocolim E_i$. We have a map $E \to F$ as the sequence maps to $F$. For any integer $i$, the natural map $\colim \Hom{\mathsf{T}}{\Sigma^i G}{E_i}{F} \xrightarrow{\sim} \Hom{\mathsf{T}}{\Sigma^i G}{E}{F}$ is an isomorphism by \cite[Lemma 2.8]{Neeman:1996}. So, combined with the observation at the beginning of this paragraph, we have that for any integer $i$, $\Hom{\mathsf{T}}{\Sigma^i G}{-}$ sends $E \to F$ to an isomorphism. As $G$ is a compact generator, this implies that the map $E \to F$ is an isomorphism. 

    Finally, as $F \cong E = \hocolim E_n$ and $E_n \in \overline{\langle G \rangle}^{[-N, \infty)}$ for all $n$, we get that $F \in \overline{\langle G \rangle}^{[-N-1, \infty)}$, and hence $\mathsf{U} \cap \mathsf{T}^+_c \subseteq \overline{\langle G \rangle}^{[-N-1, \infty)}$. By \cite[Proposition 1.9]{Neeman:2021a}, this implies that $\mathsf{U} \cap \mathsf{T}^c \subseteq \langle G \rangle^{[-N-1, \infty)}$. Conversely, by \Cref{Definition Co-approx}, $\Sigma^{-i} G \in \mathsf{U}$ for $i>>0$, and hence there exists $n \geq 0$ such that $\langle G \rangle^{[-n,\infty)} \subseteq \mathsf{U} \cap \mathsf{T}^c$, which completes the proof.
\end{proof}

We end by some recollection of standard facts and definitions related to admissible subcategories, semiorthogonal decompositions, and localisation sequences.

\subsection*{Localisations, recollements, admissible subcategories, and semiorthogonal decompositions}\label{Section recollements sods admissible categories}
In this subsection, we recall some concepts related to localisations of triangulated categories, and the existence of adjoints. 
We begin with the following well known result. For a reference, see \cite[Chapter 9, page 311-318]{Neeman:2001}.
\begin{theorem}\label{Theorem adjoint on left iff adjoint on right}
    Let $\mathsf{T}$ be a triangulated category, and $i_*: \mathsf{U} \to \mathsf{T}$ a fully faithful functor. Let the Verdier quotient of $\mathsf{T}$ by the essential image of $i_*$ be $\mathsf{V} \colonequals \mathsf{T}/i_*(\mathsf{U})$, with the localisation functor $j^* : \mathsf{T} \to \mathsf{V}$. Then,
    \begin{itemize}
        \item $i_*$ has a right adjoint $i^!$ if and only if $j^*$ has right adjoint $j_*$. Further, if the right adjoints exist, then $j_*$ is fully faithful, and $i^!$ is a Verdier localisation functor, identifying $\mathsf{U}$ with $\mathsf{T}/j_*(\mathsf{V})$.
        \item $i_*$ has a left adjoint $i^*$ if and only if $j^*$ has a left adjoint $j_!$. Further, if the left adjoints exist, then $j_!$ is fully faithful, and $i^*$ is a Verdier localisation functor, identifying $\mathsf{U}$ with $\mathsf{T}/j_!(\mathsf{V})$.
    \end{itemize}
\end{theorem}
\begin{definition}\label{Definition of localisation sequence}
    Let $\mathsf{U}$, $\mathsf{T}$ and $\mathsf{V}$ be a triangulated category with functors,
    \begin{equation}\label{Diagram Quotient sequence}
         \mathsf{U} \xrightarrow{i_*} \mathsf{T} \xrightarrow{j^*} \mathsf{V}
    \end{equation}
    with $i_*$ a fully faithful functor and $j^*$ a Verdier localisation functor, identifying $\mathsf{V}$ with $\mathsf{T}/i_*(\mathsf{U})$. Then, we say, 
    \begin{itemize}
        \item \Cref{Diagram Quotient sequence} is a \emph{localisation sequence} if $i_*$ (or, equivalently $j^*$) has a right adjoint. This gives us the diagram,
        \[\begin{tikzcd}
	   \mathsf{U} && \mathsf{T} && \mathsf{V}
	   \arrow["i_*",from=1-1, to=1-3]
	   \arrow["i^!",bend left, from=1-3, to=1-1]
	   \arrow["j^*",from=1-3, to=1-5]
	   \arrow["j_*",bend left, from=1-5, to=1-3]
        \end{tikzcd}\]
        For any $t \in \mathsf{T}$, we get a triangle $i_*i^!t \to t \to j_*j^* t \to \Sigma i_*i^!t$ which show that $(i_*\mathsf{U}, \Sigma j_*\mathsf{V})$ is both a t-structure and a co-t-structure on $\mathsf{T}$.
        \item \Cref{Diagram Quotient sequence} is a \emph{colocalisation sequence} if $i_*$ (or, equivalently $j^*$) has a left adjoint. This gives us the diagram,
        \[\begin{tikzcd}
	   \mathsf{U} && \mathsf{T} && \mathsf{V}
	   \arrow["i_*"',from=1-1, to=1-3]
	   \arrow["i^*"',bend right, from=1-3, to=1-1]
	   \arrow["j^*"',from=1-3, to=1-5]
	   \arrow["j_!"',bend right, from=1-5, to=1-3]
        \end{tikzcd}\]
        For any $t \in \mathsf{T}$, we get a triangle, $j_!j^* t \to t \to i_*i^* t \to \Sigma j_!j^* t$ which show that $(j_! \mathsf{V}, \Sigma i_* \mathsf{U})$ is both a t-structure and a co-t-structure on $\mathsf{T}$.
        \item \Cref{Diagram Quotient sequence} is a recollement if it is both a localising and a colocalising sequence. We get the following diagram in such a case, 
   \[\begin{tikzcd}
	   \mathsf{U} && \mathsf{T} && \mathsf{V}
	   \arrow["i_*",from=1-1, to=1-3]
	   \arrow["i^!",bend left, from=1-3, to=1-1]
	   \arrow["i^*"',bend right, from=1-3, to=1-1]
	   \arrow["j^*",from=1-3, to=1-5]
	   \arrow["j_*",bend left, from=1-5, to=1-3]
	   \arrow["j_!"',bend right, from=1-5, to=1-3]
    \end{tikzcd}\]
   with each functor left adjoint to the one beneath it.
    \end{itemize}
\end{definition}

We now define admissible subcategories.
\begin{definition}
     Let $\mathsf{T}$ be a triangulated category with a strictly full triangulated subcategory $\mathsf{U}$. Then, 
         \begin{itemize}
             \item $\mathsf{U}$ is \emph{right admissible} if it is an aisle, see \Cref{Remark aisle and coaisle}. Note that this gives us the localisation sequence $\mathsf{U} \to \mathsf{T} \to \mathsf{U}^{\perp}$. Conversely, given a localisation sequence $\mathsf{U} \xrightarrow{i_*} \mathsf{T} \xrightarrow{j^*} \mathsf{V}$, the essential image of $i_*$ is a right admissible category.
             \item $\mathsf{U}$ is \emph{left admissible} if it is a coaisle, see 
             \Cref{Remark aisle and coaisle}. Note that this gives us the localisation sequence $^{\perp}\mathsf{U} \to \mathsf{T} \to \mathsf{U}$. Conversely, given a localisation sequence $\mathsf{U} \xrightarrow{i_*} \mathsf{T} \xrightarrow{j^*} \mathsf{V}$, the essential image of $j_*$ is a left admissible category, where $j_*$ is the right adjoint to $j^*$.
             \item $\mathsf{U}$ is \emph{admissible} if it is both left and right admissible. In this case, we get a recollement
             \[\begin{tikzcd}
	           \mathsf{U} && \mathsf{T} && \mathsf{V}
	           \arrow["i_*",from=1-1, to=1-3]
	           \arrow["i^!",bend left, from=1-3, to=1-1]
	             \arrow["i^*"',bend right, from=1-3, to=1-1]
	             \arrow["j^*",from=1-3, to=1-5]
	           \arrow["j_*",bend left, from=1-5, to=1-3]
	           \arrow["j_!"',bend right, from=1-5, to=1-3]
            \end{tikzcd}\]
            with $\mathsf{V} \cong \mathsf{U}^{\perp} \cong {}^{\perp}\mathsf{U}$. 
         \end{itemize}
\end{definition}
We now come to the related notion of semiorthogonal decompositions.
\begin{definition}
    Let $\mathsf{T}$ be a triangulated category. A semiorthogonal decomposition on $\mathsf{T}$ is a sequence of strictly full triangulated subcategories $\mathsf{U}_1,\cdots ,\mathsf{U}_n$ such that,
    \begin{itemize}
        \item $\HomT{\mathsf{U}_i}{\mathsf{U}_j} = 0$ for all $n \geq i > j \geq 1$.
        \item The smallest triangulated subcategory of $\mathsf{T}$ containing $\bigcup_{1\leq i \leq n}\mathsf{U}_i$ is $\mathsf{T}$ itself.
    \end{itemize}
    We denote a semiorthogonal decomposition by $\langle \mathsf{U}_1, \cdots , \mathsf{U}_n \rangle$.
\end{definition}
\begin{remark}
    We will mostly be interested in semiorthogonal decompositions with only two terms. The following are easy to check,
    \begin{itemize}
        \item If $\langle \mathsf{V},\mathsf{U}\rangle $ is a semiorthogonal decompositions, then $\mathsf{V}$ is left admissible and $\mathsf{U}$ is right admissible.
        \item $\mathsf{U}$ is right admissible if and only if $\langle \mathsf{U}^{\perp},\mathsf{U} \rangle$ is a semiorthogonal decomposition.
        \item $\mathsf{V}$ is left admissible if and only if $\langle\mathsf{V}, \ \hspace{-0.2cm}^{\perp}\mathsf{V} \rangle$ is a semiorthogonal decomposition.
    \end{itemize}
\end{remark}

\section{Main results}\label{Section Main Results}

We begin by recalling the following definition from \cite{Sun/Zhang:2021}.
\begin{definition}
    \cite[Definition 4.1]{Sun/Zhang:2021} 
    Let $\mathsf{S}$ and $\mathsf{T}$ be triangulated categories with good metrics (\cite[Definition 10]{Neeman:2020}) $\{\mathcal{M}_n\}$ and $\{\mathcal{N}_n\}$ respectively. Let $F : \mathsf{S} \to \mathsf{T}$ a triangulated functor. Then, we say $F$ is a \emph{compression functor} if for all $i > 0$, there exists a $n > 0$ such that $F(\mathcal{M}_n) \subseteq \mathcal{N}_i$.
\end{definition}

In similar spirit, we define the following notions.

\begin{definition}\label{Definition compressed metrics}
    Let $\mathsf{T}$ be a $R$-linear triangulated category for a commutative ring $R$. Then, 
    \begin{enumerate}
        \item  We say an extended good metric $\mathcal{M}$ (\Cref{Definition Metric}) is a \emph{compressed metric} if all triangulated functors $F : \mathsf{T} \to \mathsf{T}$ are compression functors with respect to $\mathcal{M}$. That is, for all $i \in \mathbb{Z}$, there exists $n \in \mathbb{Z}$ such that $F(\mathcal{M}_n)\subseteq \mathcal{M}_i$. 
        
        We further say an extended good metric $\mathcal{M}$ (\Cref{Definition Metric}) is a \emph{$\mathbb{N}$-compressed metric} if for  any triangulated functors $F : \mathsf{T} \to \mathsf{T}$, there exists an integer $l \geq 0$ such that $F(\mathcal{M}_i)\subseteq \mathcal{M}_{i+l}$ for all $i \in \mathbb{Z}$.

        \item A finite generating sequence (\Cref{Definition generating sequence}) $\mathcal{G}$ is \emph{compressed} (resp.\ \emph{$\mathbb{N}$-compressed}) if the metric $\mathcal{R}^{\mathcal{G}} \cap \mathsf{T}^c$ is a compressed metric (resp.\ $\mathbb{N}$-compressed metric), and is equivalent to the metric $\mathcal{M}^{\mathcal{G}}$ on $\mathsf{T}^c$, see \Cref{Definition generating sequence} for the definition of $\mathcal{R}^{\mathcal{G}}$ and $\mathcal{M}^{\mathcal{G}}$.
        \item We say that a Serre subcategory $\mathcal{C} \subseteq \prod_{i \in \mathbb{Z}}\operatorname{Mod}(R)$ is \emph{compressed} if it is stable under sending the indexing sequence $\{i\}_{i\in \mathbb{Z}}$ to $\{i+n\}_{i \in \mathbb{Z}}$ for any integer $n$. That is, if $\{M_i\}_{i \in \mathbb{Z}}$ is in $\mathcal{C}$ and $n \in \mathbb{Z}$, then $\{M_{i + n}\}_{i \in \mathbb{Z}}$ is also in $\mathcal{C}$.
    \end{enumerate}
   
\end{definition}
The following result gives an important class of generating sequences and metrics which are compressed.

\begin{theorem}\label{Theorem Examples of Compressed metrics}
    Let $\mathsf{S}$ be a triangulated category with a classical generator $G$, that is $\mathsf{S} = \langle G \rangle$. Then, the following finite generating filtrations are $\mathbb{N}$-compressed, see \Cref{Definition compressed metrics}(2),
    \begin{itemize}
        \item $\mathcal{G}^n = \{\Sigma^{-n}G \}$ for all $n \in \mathbb{Z}$.
        \item $\mathcal{G}'^n = \{\Sigma^n G \}$ for all $n \in \mathbb{Z}$.
    \end{itemize}
\end{theorem}
\begin{proof}
    Let $F:\mathsf{S} \to \mathsf{S}$ be any triangulated functor. Then, $F(G) \in \mathsf{S} = \langle G \rangle$. But, $\langle G \rangle = \bigcup_{i \geq 0} \langle G \rangle^{[-i,i]}$, and so, there exists $i \geq 0$ such that $F(G) \in \langle G \rangle^{[-i,i]}$, see \Cref{Notation from Neeman}. As $F$ is a triangulated functor, it preserves direct sums, summands, shifts, and extensions. So, for all $n \geq 0$, we have that
    \[ F(\mathcal{M}^{\mathcal{G}}_{n+i}) \subseteq \mathcal{M}^{\mathcal{G}}_n \text{ and } F(\mathcal{M}^{\mathcal{G}'}_{n+i}) \subseteq \mathcal{M}^{\mathcal{G}'}_n\]
    with notation as in \Cref{Definition generating sequence},
    which shows that both the metrics are $\mathbb{N}$-compressed. But, by \Cref{Example generating sequence approx/co-approx}, $\mathcal{R}^{\mathcal{G}} \cap \mathsf{T}^c$ is $\mathbb{N}$-equivalent (\Cref{Definition Metric}) to $\mathcal{M}^{\mathcal{G}}$ and $\mathcal{R}^{\mathcal{G}'} \cap \mathsf{T}^c$ is $\mathbb{N}$-equivalent to $\mathcal{M}^{\mathcal{G}'}$ and hence these metrics are also $\mathbb{N}$-compressed, which in turn implies that the generating sequences are $\mathbb{N}$-compressed by \Cref{Definition compressed metrics}.
\end{proof}

The following result is known in the literature, but we give a short proof.

\begin{theorem}\label{Theorem localisation on Tc gives recollement on T}
    Let $\mathsf{T}$ be a compactly generated triangulated category, and suppose there exists a semiorthogonal decomposition $\langle\mathsf{A},\mathsf{B}\rangle$ on $\mathsf{T}^c$. We know this gives a colocalisation sequence on $\mathsf{T}^c$ as follows,
      \[\begin{tikzcd}
	   \mathsf{A} && \mathsf{T}^c && \mathsf{B}
	   \arrow["i_*"',from=1-1, to=1-3]
	   \arrow["i^*"',bend right, from=1-3, to=1-1]
	   \arrow["j^*"',from=1-3, to=1-5]
	   \arrow["j_!"',bend right, from=1-5, to=1-3]
	\end{tikzcd}\]
    Then, this extends to a recollement on $\mathsf{T}$ as follows,
        \[\begin{tikzcd}
	   \mathsf{T}_{\mathsf{A}} && \mathsf{T} && \mathsf{T}_{\mathsf{B}}
	   \arrow["i_*",from=1-1, to=1-3]
	   \arrow["i^!",bend left, from=1-3, to=1-1]
	   \arrow["i^*"',bend right, from=1-3, to=1-1]
	   \arrow["j^*",from=1-3, to=1-5]
	   \arrow["j_*",bend left, from=1-5, to=1-3]
	   \arrow["j_!"',bend right, from=1-5, to=1-3]
    \end{tikzcd}\]
    where $\mathsf{T}_{\mathsf{A}} = \operatorname{Coprod}(\mathsf{A})$ and $\mathsf{T}_{\mathsf{B}} = \operatorname{Coprod}(\mathsf{B})$.
    In particular, $\langle\mathsf{A}^{\perp},\mathsf{B}^{\perp}\rangle$ is a semiorthogonal decomposition on $\mathsf{T}$.
\end{theorem}
\begin{proof}
    It is well known, and easy to show, that the colocalisation sequence on $\mathsf{T}^c$ extends to one on $\mathsf{T}$ as follows,
    \[\begin{tikzcd}
	   \mathsf{T}_{\mathsf{A}} && \mathsf{T} && \mathsf{T}_{\mathsf{B}}
	   \arrow["i_*"',from=1-1, to=1-3]
	   \arrow["i^*"',bend right, from=1-3, to=1-1]
	   \arrow["j^*"',from=1-3, to=1-5]
	   \arrow["j_!"',bend right, from=1-5, to=1-3]
	\end{tikzcd}\]
    We need to show that it is also a localisation sequence. It is enough to show that there exists a right adjoint to $i_*$. Note that $\mathsf{T}_{\mathsf{A}}$ is a compactly generated triangulated category, and $i_*$ preserves coproducts. So, by Neeman's adjoint functor theorem (see \cite[Theorem 4.1]{Neeman:1996}), $i_*$ has a right adjoint giving us the required recollement. 

     In particular, $\langle j_*(\mathsf{T}_{\mathsf{B}}),i_*(\mathsf{T}_{\mathsf{A}})\rangle$ is a semiorthogonal decomposition on $\mathsf{T}$. So, $j_*(\mathsf{T}_{\mathsf{B}}) = i_*(\mathsf{T}_{\mathsf{A}})^{\perp} = \operatorname{Coprod}(\mathsf{A})^{\perp} = \mathsf{A}^{\perp}$. As $\langle i_*(\mathsf{T}_{\mathsf{A}}),j_!(\mathsf{T}_{\mathsf{B}})\rangle$ is also a semiorthogonal decomposition on $\mathsf{T}$, we have that $i_*(\mathsf{T}_{\mathsf{A}}) = j_!(\mathsf{T}_{\mathsf{B}})^{\perp} = \operatorname{Coprod}(\mathsf{B})^{\perp} = \mathsf{B}^{\perp}$. And so, this shows that $\langle \mathsf{A}^{\perp},\mathsf{B}^{\perp} \rangle$ is a semiorthogonal decomposition on $\mathsf{T}$.
\end{proof}

Now, we get to the first main result of this section.
\begin{theorem}\label{Theorem extending localisation sequence from Tc to its closure}
    Let $\mathsf{T}$ a triangulated category with a generating sequence $\mathcal{G}$ (\Cref{Definition generating sequence}), and an orthogonal metric $\mathcal{R}$ (\Cref{Definition Metric}) such that for all $n \in \mathbb{Z}$, $\HomT{\mathcal{G}^n}{\mathcal{R}_i} = 0$ for all $i >> 0$. We further assume that the metric $\mathcal{R} \cap \mathsf{T}^c$ is compressed, see \Cref{Definition compressed metrics}(1). Suppose we are given a localisation sequence on $\mathsf{T}^c$ as follows,
    \[\begin{tikzcd}
	   \mathsf{A} && \mathsf{T}^c && \mathsf{B}
	   \arrow["i_*",from=1-1, to=1-3]
	   \arrow["i^!",bend left, from=1-3, to=1-1]
	   \arrow["j^*",from=1-3, to=1-5]
	   \arrow["j_*",bend left, from=1-5, to=1-3]
	\end{tikzcd}\]
    with $i_* : \mathsf{A} \to \mathsf{T}^c$ and $j_* : \mathsf{B} \to \mathsf{T}^c$ inclusions of strictly full subcategories.
    Let $\mathsf{T}_{\mathsf{A}}$ and $\mathsf{T}_{\mathsf{B}}$ be the localising subcategories generated by $\mathsf{A}$ and $\mathsf{B}$ respectively inside $\mathsf{T}$, that is, $\mathsf{T}_{\mathsf{A}} = \operatorname{Coprod}(\mathsf{A})$ and $\mathsf{T}_{\mathsf{B}} = \operatorname{Coprod}(\mathsf{B})$, see \Cref{Notation from Neeman}. These categories have generating sequences $\mathcal{G}_{\mathsf{A}} \colonequals i^!(\mathcal{G})$ and $\mathcal{G}_{\mathsf{B}} \colonequals j^*(\mathcal{G})$ and orthogonal metrics $\mathcal{R}^{\mathsf{A}} \colonequals i^!(\mathcal{R})$ and $\mathcal{R}^{\mathsf{B}} \colonequals j^*(\mathcal{R})$ on $\mathsf{T}_{\mathsf{A}}$ and $\mathsf{T}_{\mathsf{B}}$ respectively.

    Then, we get a localisation sequence on the closure of compacts $\overline{\mathsf{T}^c}$ (see \Cref{Definition closure of compacts}) as follows,
    \[\begin{tikzcd}
	   \overline{\mathsf{T}^{c}_{\mathsf{A}}} && \overline{\mathsf{T}^c}  && \overline{\mathsf{T}^{c}_{\mathsf{B}}}
	   \arrow["i_*",from=1-1, to=1-3]
	   \arrow["i^!",bend left, from=1-3, to=1-1]
	   \arrow["j^*",from=1-3, to=1-5]
	   \arrow["j_*",bend left, from=1-5, to=1-3]
	\end{tikzcd}\]
    with the closure of compacts on $\mathsf{T}_{\mathsf{A}}$, $\mathsf{T}$, and $\mathsf{T}_{\mathsf{B}}$ with respect to the metrics $\mathcal{R}^{\mathsf{A}}$, $\mathcal{R}$, and $\mathcal{R}^{\mathsf{B}}$ respectively. Further,  $\overline{\mathsf{T}^{c}_{\mathsf{A}}} = \mathsf{T}_{\mathsf{A}} \cap \overline{\mathsf{T}^c}$ and $\overline{\mathsf{T}^{c}_{\mathsf{B}}} = \mathsf{T}_{\mathsf{B}} \cap \overline{\mathsf{T}^c}$.
\end{theorem}
 
\begin{proof}
    Note that, as can be easily shown, we get a localisation sequence on $\mathsf{T}$ as follows,
    \[\begin{tikzcd}
	   \mathsf{T}_{\mathsf{A}} && \mathsf{T}  && \mathsf{T}_{\mathsf{B}}
	   \arrow["i_*",from=1-1, to=1-3]
	   \arrow["i^!",bend left, from=1-3, to=1-1]
	   \arrow["j^*",from=1-3, to=1-5]
	   \arrow["j_*",bend left, from=1-5, to=1-3]
	\end{tikzcd}\]
    where we denote the functors by the same symbols as the localisation sequence on $\mathsf{T}^c$. 

    We observe now that all the four functors involved in the localisation sequence on $\mathsf{T}$ preserve coproducts, and hence respect homotopy colimits. The functors $i_\ast$ and $j^\ast$ preserve coproducts as they have right adjoints. Further, the functors $i^!$ and $j_\ast$ preserve coproducts by \cite[Theorem 5.1]{Neeman:1996} because their left adjoints preserve compact objects.

    We now show that this localisation sequence restricts to the closure of compacts.
    That is, for any $F \in \overline{\mathsf{T}^c}$, we want to show that $i^!(F) \in \overline{\mathsf{T}^c_{\mathsf{A}}}$ and $j^*(F) \in \overline{\mathsf{T}^c_{\mathsf{B}}}$. As $F \in \overline{\mathsf{T}^c}$, there is a sequence $E_i$ mapping to $F$ such that $\hocolim E_i \cong F$ and $\operatorname{Cone}(E_i \to E_{i+j}) \in \mathcal{R}_{i+1}$ for all $i,j \geq 0$ by \Cref{Lemma approximating sequence for closure of compacts}. So, we get that $i^!(F) = \hocolim i^!(E_i)$ and $j^\ast = \hocolim j^\ast(E_i)$ as $i^!$ and $j^\ast(F)$ respect homotopy colimits. Note that $\operatorname{Cone}(i^!(E_i) \to i^{!}(E_{i+j})) \in \mathcal{R}^{\mathsf{A}}_{i+1}$ and $\operatorname{Cone}(j^\ast(E_i) \to j^\ast(E_{i+j})) \in \mathcal{R}^{\mathsf{B}}_{i+1}$ for all $i,j \geq 0$ by definition. Therefore, by \Cref{Lemma approximating sequence for closure of compacts}, we get that $i^!(F) = \hocolim i^!(E_i) \in \overline{\mathsf{T}^c_{\mathsf{A}}}$ and $j^\ast(F) = \hocolim j^\ast(E_i) \in \overline{\mathsf{T}^c_{\mathsf{B}}}$, which is what we needed to show.
    
    Finally, we need to show that $i_\ast\big(\overline{\mathsf{T}^{c}_{\mathsf{A}}}\big) \subseteq \overline{\mathsf{T}^c} $ and $j_\ast\big(\overline{\mathsf{T}^{c}_{\mathsf{B}}}\big) \subseteq \overline{\mathsf{T}^c}$. We show that $i_\ast\big(\overline{\mathsf{T}^{c}_{\mathsf{A}}}\big) \subseteq \overline{\mathsf{T}^c}$, and the other inclusion follows similarly. Let $A \in \overline{\mathsf{T}^{c}_{\mathsf{A}}}$. Then, there exists a sequence $A_*$ mapping to $A$ such that $\hocolim A_i \cong A$, and $\operatorname{Cone}(A_i \to A_{i+j}) \in \mathcal{R}^{\mathsf{A}}_{i+1} \cap \mathsf{T}^c$ for all $i,j \geq 1$. But, $\mathcal{R} \cap \mathsf{T}^c$ is assumed to be a compressed metric, and $\mathcal{R}^{\mathsf{A}}_i = i_*i^!(\mathcal{R}_i)$, that is, it is the image of $\mathcal{R}$ under the endofunctor $i_*i^!$. So, by the definition of compressed metric (\Cref{Definition compressed metrics}(1)), for all $j \geq 0$, there exists a $n \geq 0$ such that  $i_*i^!(\mathcal{R}^{\mathsf{A}}_n \cap \mathsf{T}^c) \subseteq \mathcal{R}_j \cap \mathsf{T}^c$, which gives us that $\hocolim A_i \in \overline{\mathsf{T}^c}$ by \Cref{Lemma approximating sequence for closure of compacts}.
\end{proof}
This gives us the following corollaries.

\begin{corollary}\label{Corollary admissible categories from Tc to T-c}
    Let $\mathsf{T}$ be a compactly generated triangulated category, with a single compact generator $G$ such that $\HomT{G}{\Sigma^i G} = 0$ for all $i>>0$. Suppose we have the localisation sequence,  
     \[\begin{tikzcd}
	   \mathsf{A} && \mathsf{T}^c && \mathsf{B}
	   \arrow["i_*",from=1-1, to=1-3]
	   \arrow["i^!",bend left, from=1-3, to=1-1]
	   \arrow["j^*",from=1-3, to=1-5]
	   \arrow["j_*",bend left, from=1-5, to=1-3]
	\end{tikzcd}\]
    with $i_* : \mathsf{A} \to \mathsf{T}^c$ and $j_* : \mathsf{B} \to \mathsf{T}^c$ inclusions of strictly full subcategories.
    Let $\mathsf{T}_{\mathsf{A}}$ and $\mathsf{T}_{\mathsf{B}}$ be the localising subcategories of $\mathsf{T}$ 
    generated by $\mathsf{A}$ and $\mathsf{B}$ respectively, that is, $\mathsf{T}_{\mathsf{A}} = \operatorname{Coprod}(\mathsf{A})$ and $\mathsf{T}_{\mathsf{B}} = \operatorname{Coprod}(\mathsf{B})$, see \Cref{Notation from Neeman}. Then, we get the following localisation sequence,
     \[\begin{tikzcd}
	   (\mathsf{T}_{\mathsf{A}})^-_c && \mathsf{T}^-_c  && (\mathsf{T}_{\mathsf{B}})^-_c 
	   \arrow["i_*",from=1-1, to=1-3]
	   \arrow["i^!",bend left, from=1-3, to=1-1]
	   \arrow["j^*",from=1-3, to=1-5]
	   \arrow["j_*",bend left, from=1-5, to=1-3]
	\end{tikzcd}\]
    with the categories $\mathsf{T}^-_c$,  $(\mathsf{T}_{\mathsf{A}})^-_c$, and $(\mathsf{T}_{\mathsf{B}})^-_c$ defined as in \Cref{Convention closure of compacts for approx/co-approx} with respect to any t-structure in the preferred quasiequivalence class of t-structures on $\mathsf{T}$, $\mathsf{T}_{\mathsf{A}}$, and $\mathsf{T}_{\mathsf{B}}$ respectively (see \Cref{Definition preferred equivalence class}).
\end{corollary}
\begin{proof}
    We are working with the generating sequence $\mathcal{G}$ given by $\mathcal{G}^n = \{\Sigma^{-n}G\}$, which is compressed, see \Cref{Theorem Examples of Compressed metrics}. Hence the metric $\mathcal{R}^{\mathcal{G}} \cap \mathsf{T}^c$ is compressed by \Cref{Definition compressed metrics}.
    By \Cref{Convention closure of compacts for approx/co-approx}, $\mathsf{T}^-_c$ is the closure of compacts $\overline{\mathsf{T}^c}$, see \Cref{Definition closure of compacts}.  Further, for any $n \in \mathbb{Z}$, we have that $ \HomT{\mathcal{G}^n}{\mathcal{R}^{\mathcal{G}}_i} = \HomT{\Sigma^{-n} G}{\overline{\langle G \rangle}^{(-\infty,-i]}} = 0$ for $i > > 0$ as $G$ is compact, and we have $\HomT{G}{\Sigma^i G} = 0$ for $i > > 0$ by hypothesis. So, we get the required result from \Cref{Theorem extending localisation sequence from Tc to its closure}.
\end{proof}    

\begin{corollary}\label{Corollary admissible categories from Tc to T+c}
    Let $\mathsf{T}$ be a triangulated category with a single compact generator $G$ such that $\HomT{G}{G[-i]} = 0$ for all $i>>0$. Suppose we have the localisation sequence, 
     \[\begin{tikzcd}
	   \mathsf{A} && \mathsf{T}^c && \mathsf{B}
	   \arrow["i_*",from=1-1, to=1-3]
	   \arrow["i^!",bend left, from=1-3, to=1-1]
	   \arrow["j^*",from=1-3, to=1-5]
	   \arrow["j_*",bend left, from=1-5, to=1-3]
	\end{tikzcd}\]
    with $i_* : \mathsf{A} \to \mathsf{T}^c$ and $j_* : \mathsf{B} \to \mathsf{T}^c$ inclusions of strictly full subcategories.
    Let $\mathsf{T}_{\mathsf{A}}$ and $\mathsf{T}_{\mathsf{B}}$ be the localising subcategories of $\mathsf{T}$ generated by $\mathsf{A}$ and $\mathsf{B}$ respectively, that is, $\mathsf{T}_{\mathsf{A}} = \operatorname{Coprod}(\mathsf{A})$ and $\mathsf{T}_{\mathsf{B}} = \operatorname{Coprod}(\mathsf{B})$, see \Cref{Notation from Neeman}. Then, we get the following localisation sequence 
    \[\begin{tikzcd}
	   (\mathsf{T}_{\mathsf{A}})^+_c && \mathsf{T}^+_c  && (\mathsf{T}_{\mathsf{B}})^+_c 
	   \arrow["i_*",from=1-1, to=1-3]
	   \arrow["i^!",bend left, from=1-3, to=1-1]
	   \arrow["j^*",from=1-3, to=1-5]
	   \arrow["j_*",bend left, from=1-5, to=1-3]
	\end{tikzcd}\]
    with the categories $\mathsf{T}^+_c$,  $(\mathsf{T}_{\mathsf{A}})^+_c$, and $(\mathsf{T}_{\mathsf{B}})^+_c$ defined as in \Cref{Convention closure of compacts for approx/co-approx} with respect to any co-t-structure in the preferred quasiequivalence class of co-t-structures on $\mathsf{T}$, $\mathsf{T}_{\mathsf{A}}$, and $\mathsf{T}_{\mathsf{B}}$ respectively (see \Cref{Definition preferred equivalence class}).
\end{corollary}
\begin{proof}
     We are working with the compressed generating sequence $\mathcal{G}$ given by $\mathcal{G}^n = \{\Sigma^{n}G\}$.
    Hence the metric $\mathcal{R}^{\mathcal{G}} \cap \mathsf{T}^c$ is compressed by \Cref{Definition compressed metrics}.
    By \Cref{Convention closure of compacts for approx/co-approx}, $\mathsf{T}^+_c$ is the closure of compacts $\overline{\mathsf{T}^c}$, see \Cref{Definition closure of compacts}.  Further, for any $n \in \mathbb{Z}$, we have that $\HomT{\Sigma^{-n} G}{\overline{\langle G \rangle}^{[i,\infty)}} = 0$ for $i>>0$ as $G$ is compact, and we have $\HomT{G}{\Sigma^i G} = 0$ for $i > > 0$ by hypothesis. As $\mathcal{R}_n^{\mathcal{G}} = \Sigma^{-n} \mathsf{U}_G \subseteq \overline{\langle G \rangle}^{[n-1,\infty)}$ by \cite[Theorem 2.3.4]{Bondarko:2022} where $(\mathsf{U}_G, \mathsf{V}_G)$ is the co-t-structure generated by $G$, $\HomT{\mathcal{G}^n}{\mathcal{R}^{\mathcal{G}}_i}$ for $i > > 0$. So, we get the required result from \Cref{Theorem extending localisation sequence from Tc to its closure}.
\end{proof}
We now want to prove the corresponding results about $\mathsf{T}^b_c$, see \Cref{Definition closure of compacts}. For that, we need a few more ingredients. We start with some notation.

\begin{notation}\label{Notation for relative orthogonal}
        Let $\mathsf{U}$ and $\mathsf{S}$ be full subcategory of $\mathsf{T}$, then the full subactegory of $\mathsf{S}$ of objects which are right (resp. left) orthogonal to $\mathsf{U}$ is denoted by $\mathsf{U}^{\perp}_{\mathsf{S}}$ (resp. $^{\perp} \mathsf{U}_{\mathsf{S}}$). If we write $\mathsf{U}^{\perp}$ or $^{\perp} \mathsf{U}$ without a subscript, it is assumed that $\mathsf{S} = \mathsf{T}$.
\end{notation}

We now give a definition, and then prove a result which is similar in spirit to \cite[Theorem 3.1.1.II.2]{Bondarko:2023}.
%\Kabeer{Work on things from here.}

\begin{definition}\label{Definition Categories from compressed sequence and category}
     Let $\mathsf{T}$ be a compactly generated $R$-linear triangulated category with a generating sequence $\mathcal{G}$, and a compressed abelian subcategory $\mathcal{S}\subseteq \prod_{i \in \mathbb{Z}}\operatorname{Mod}(R)$, see \Cref{Definition compressed metrics}. We define the following thick subcategory of $\mathsf{T}$,
         \[\mathsf{T}_{\mathcal{S}} \colonequals \Big\{t \in \mathsf{T} : \big\{\HomT{h_i}{t}\big\}_{i \in \mathbb{Z}} \in \mathcal{S} \text{ for all sequences } \{h_i\}_{i \in \mathbb{Z}} \text{ with } h_i \in \mathcal{M}_{i}^{\mathcal{G}}\Big\}\]
     see \Cref{Definition generating sequence}.
\end{definition}

\begin{theorem}\label{Corollary Bondarko type result}
     Let $\mathsf{T}$ be a compactly generated $R$-linear triangulated category with a $\mathbb{N}$-compressed generating sequence $\mathcal{G}$, and a compressed abelian subcategory $\mathcal{S}\subseteq \prod_{i \in \mathbb{Z}}\operatorname{Mod}(R)$, see \Cref{Definition compressed metrics}. 
     Then, for any semiorthogonal decomposition $\langle \mathsf{A},\mathsf{B} \rangle$ of $\mathsf{T}^c$, we have that $\langle\mathsf{A}^{\perp}_{\mathsf{T}_{\mathcal{S}}},\mathsf{B}^{\perp}_{\mathsf{T}_{\mathcal{S}}}\rangle$ is a semiorthogonal decompositions of $\mathsf{T}_{\mathcal{S}}$, see \Cref{Notation for relative orthogonal} and \Cref{Definition Categories from compressed sequence and category}. Further, $\mathsf{B}^{\perp}_{\mathsf{T}_{\mathcal{S}}} = \operatorname{Coprod}(\mathsf{A}) \cap \mathsf{T}_{\mathcal{S}}$.
\end{theorem}

\begin{proof}
    By \Cref{Theorem localisation on Tc gives recollement on T}, we get the recollement,
    \[\begin{tikzcd}
	   \mathsf{T}_{\mathsf{A}} && \mathsf{T} && \mathsf{T}_{\mathsf{B}}
	   \arrow["i_*",from=1-1, to=1-3]
	   \arrow["i^!",bend left, from=1-3, to=1-1]
	   \arrow["i^*"',bend right, from=1-3, to=1-1]
	   \arrow["j^*",from=1-3, to=1-5]
	   \arrow["j_*",bend left, from=1-5, to=1-3]
	   \arrow["j_!"',bend right, from=1-5, to=1-3]
    \end{tikzcd}\]
    where $\mathsf{T}_{\mathsf{A}} = \operatorname{Coprod}(\mathsf{A})$ and $\mathsf{T}_{\mathsf{B}} = \operatorname{Coprod}(\mathsf{B})$, see \Cref{Notation from Neeman}.
    
    We know that $\mathsf{B}^{\perp} = j_!(\mathsf{T}_{\mathsf{B}})^{\perp} = i_*(\mathsf{T}_{\mathsf{A}})$ is a right admissible category of $\mathsf{T}$, with the corresponding left admissible subcategory $\mathsf{A}^{\perp} = i_*(\mathsf{T}_{\mathsf{A}})^{\perp} = j_*(\mathsf{T}_{\mathsf{B}})$. We need to show that the corresponding localisation sequence restricts to $\mathsf{T}_{\mathcal{S}}$.
    
    That is, we need to show that for any $t \in$ $\mathsf{T}_{\mathcal{S}}$, we have that $i_*i^!(t), j_*j^*(t) \in \mathsf{T}_{\mathcal{S}}$. So, let  $\{h_n\}_{n \in \mathbb{Z}}$ be any collection of objects such that $h_n \in \mathcal{M}_{n}^{\mathcal{G}}$ for all $n \in \mathbb{Z}$. As $\mathcal{G}$ is a $\mathbb{N}$-compressed generating sequence, see \Cref{Definition compressed metrics}, we get that there exists an integer $l \geq 0$ such that $i_*i^*(h_n) \in \mathcal{M}_{n+l}^{\mathcal{G}}$ for each $n \in \mathbb{Z}$.

    By adjunction, we have that $\HomT{h_n}{i_*i^!(t)} = \HomT{i_*i^*(h_n)}{t}$ for any $t \in \mathsf{T}$.
    In particular, for any $t \in \mathsf{T}_{\mathcal{S}}$, we get that, $\{\HomT{h_n}{i_*i^!(t)}\}_{n \in \mathbb{Z}} = \{\HomT{i_*i^*(h_n)}{t}\}_{n \in \mathbb{Z}}$. But the latter belongs to $\mathcal{S}$ by the above paragraph, the fact that $\mathcal{S}$ is compressed (see \Cref{Definition compressed metrics}(3)), and the definition of $\mathsf{T}_{\mathcal{S}}$ (\Cref{Definition Categories from compressed sequence and category}). And therefore we get that, $\{\HomT{h_n}{i_*i^!(t)}\}_{n \in \mathbb{Z}} \in \mathcal{S}$ for any sequence $\{h_n\}_{n \in \mathbb{Z}}$ with $h_n \in \mathcal{M}_{n}^{\mathcal{G}}$. This shows that $i_*i^!(t) \in \mathsf{T}_{\mathcal{S}}$. Similarly, we get that $j_*j^*(t) \in \mathsf{T}_{\mathcal{S}}$, which is what we needed to show.

    Finally, note that $\mathsf{B}^{\perp} = \operatorname{Coprod}(\mathsf{A})$. And so, we get that $\mathsf{B}^{\perp}_{\mathsf{T}_{\mathcal{S}}} = \operatorname{Coprod}(\mathsf{A}) \cap \mathsf{T}_{\mathcal{S}}$.
\end{proof}

\begin{corollary}\label{Corollary admissable subcategories perpendicular}
    Let $\mathsf{T}$ be a triangulated category with a $\mathbb{N}$-compressed generating sequence $\mathcal{G}$, see \Cref{Definition compressed metrics}(2). Let $\langle\mathcal{W}, \mathcal{V} \rangle$ be a semiorthogonal decomposition on $\mathsf{T}^c$. Then, $\langle \mathcal{W}^{\perp}_{\mathcal{G}^{\perp}}, \mathcal{V}^{\perp}_{\mathcal{G}^{\perp}} \rangle$ is a semiorthogonal decomposition on $\mathcal{G}^{\perp}$, see \Cref{Notation for relative orthogonal}. Further, $\mathcal{V}^{\perp}_{\mathcal{G}^{\perp}} = \operatorname{Coprod}(\mathcal{W}) \cap \mathcal{G}^{\perp}$, see \Cref{Notation from Neeman}.
\end{corollary}
\begin{proof}
    This follows from \Cref{Corollary Bondarko type result} by taking $R = \mathbb{Z}$ and $\mathcal{S}\subseteq \prod_{i \in \mathbb{Z}}\operatorname{Mod}(\mathbb{Z})$ to be all the sequences of modules $\{M_i\}_{i \in \mathbb{Z}}$ such that $M_i = 0$ for all $i >> 0$. Then, as $\mathcal{G}$ is a $\mathbb{N}$-compressed generating sequence, in particular it is a finite generating sequence, see \Cref{Definition compressed metrics}. So, $\mathcal{G}^n = \{G_{n,i}\}_{i=1}^{j_n}$ for some objects $G_{n,i} \in \mathsf{T}^c$ and $j_n \geq 0$ for all $n \in \mathbb{Z}$. Hence, $t \in \mathsf{T}_{\mathcal{S}}$ (\Cref{Definition Categories from compressed sequence and category}) if and only if $t \in \mathcal{G}^{\perp}$, as we can take the sequence $\{h_n\} \in \mathcal{S}$ with $h_n = \oplus_{i=1}^{j_{-n}} G_{-n,i}$ for all $n\in \mathbb{Z}$. And so, the required result follows from \Cref{Corollary Bondarko type result}.
\end{proof}
As an immediate corollary, we get the following result.
\begin{corollary}\label{For T^+ and T^-}
    Let $\mathsf{T}$ be triangulated category with a single compact generator $G$. Let $\langle \mathcal{W}, \mathcal{V} \rangle$ be a semiorthoginal decomposition on $\mathsf{T}^c$. Then,
    \begin{itemize}
        \item $\langle \mathcal{W}_{\mathsf{T}^+}^{\perp}, \mathcal{V}_{\mathsf{T}^+}^{\perp} \rangle$ is a semiorthogonal decomposition on $\mathsf{T}^+ \colonequals \bigcup_{n \in \mathbb{Z}}\Sigma^n\mathsf{V}$, where $(\mathsf{U}, \mathsf{V})$ is a t-structure in the preferred equivalence class of t-structures on $\mathsf{T}$, see \Cref{Definition preferred equivalence class}.
        \item $\langle\mathcal{W}_{\mathsf{T}^-}^{\perp}, \mathcal{V}_{\mathsf{T}^-}^{\perp} \rangle$ restricts to a semiorthogonal decomposition on $\mathsf{T}^- \colonequals \bigcup_{n \in \mathbb{Z}}\Sigma^n\mathsf{V}$, where $(\mathsf{U}, \mathsf{V})$ is a co-t-structure in the preferred equivalence class of co-t-structures on $\mathsf{T}$, see \Cref{Definition preferred equivalence class}.
    \end{itemize}
\end{corollary}
\begin{proof}
    We only prove (1), as (2) follows exactly the same way. Let $G$ be a compact generator for $\mathsf{T}$, which exists by assumption. Let $\mathcal{G}$ be the finite generating filtration (\Cref{Definition generating sequence}) given by $\mathcal{G}^n = \{\Sigma^{-n}G\}$.  By \Cref{Theorem Examples of Compressed metrics}, $\mathcal{G}$ is $\mathbb{N}$-compressed. Further, the category $\mathsf{T}^+ = \mathcal{G}^{\perp} = \bigcup_{n \in \mathbb{Z}} (\mathcal{M}^{\mathcal{G}}_{n})^{\perp}$, see \Cref{Definition generating sequence} for the definition of $\mathcal{M}^{\mathcal{G}}$. So, (1) follows from \Cref{Corollary admissable subcategories perpendicular}.
\end{proof}

Finally, combining the results, we get the following result for the bounded objects in the closure of compacts $\mathsf{T}^b_c$, see \Cref{Definition closure of compacts}.
\begin{theorem}\label{Theorem admissible categories and T^b_c}
     Let $\mathsf{T}$ be a triangulated category with a $\mathbb{N}$-compressed generating sequence $\mathcal{G}$ (\Cref{Definition compressed metrics}), and a orthogonal good metric (\Cref{Definition Metric}) $\mathcal{R}$ such that for all $n\in \mathbb{Z}$, $\HomT{\mathcal{G}^n}{\mathcal{R}_i} = 0$ for all $i >> 0 $. Then, for any admissible category $\mathsf{A}$ of $\mathsf{T}^c$, we have that $\operatorname{Coprod}(\mathsf{A})\cap{\mathsf{T}^b_c}$ is a right admissible subcategory of $\mathsf{T}^b_c$, see \Cref{Notation from Neeman} and \Cref{Definition closure of compacts}. 
\end{theorem}
\begin{proof}
    As $\mathsf{A}$ is admissible, we get that $\mathsf{B} = {}^{\perp}\mathsf{A}_{\mathsf{T}^c}$ is right admissible, see \Cref{Notation for relative orthogonal}. %and $\mathsf{A}^{\perp}_{\mathsf{T}^c}$ is left admissible. 
    So, by \Cref{Corollary admissable subcategories perpendicular}, $\mathsf{B}^{\perp}_{\mathcal{G}^{\perp}} = \operatorname{Coprod}(\mathsf{A})\cap \mathcal{G}^{\perp}$ is a right admissible subcategory of $\mathcal{G}^{\perp}$. As $\mathsf{A}$ is admissible, we have that $\operatorname{Coprod}(\mathsf{A})\cap \overline{\mathsf{T}^c}$ is also admissible by \Cref{Theorem extending localisation sequence from Tc to its closure}. As these are compatible, we get that $\operatorname{Coprod}(\mathsf{A})\cap{\mathsf{T}^b_c}$ is a right admissible subcategory of $\mathsf{T}^b_c$.
\end{proof}

\begin{remark}\label{Remark on paper by Sun and Zhang}
    We should remark here that results similar \Cref{Theorem admissible categories and T^b_c} have appeared in the literature, in particular \cite[Theorem 1.2]{Sun/Zhang:2021}. To compare the results, note that an admissible subcategory gives a recollement, and that $\mathsf{T}^b_c$ is the completion of $\mathsf{T}^c$ if for example $\mathsf{T}$ is weakly $\mathcal{G}$-approximable by 
    %\Kabeer{Cite correct theorem number here.}
    \cite[Corollary 8.4]{ManaliRahul:2025}. The result in \cite{Sun/Zhang:2021} has much weaker conditions on the category, but assumes that all the functors involved in the recollement are compression functors, which is not obvious if we are in the setting of  \Cref{Theorem admissible categories and T^b_c}. Further, \Cref{Theorem admissible categories and T^b_c} gives a very concrete description of the new right admissible subcategory we get, which again is not obvious to get from the result in \cite{Sun/Zhang:2021}. In any case, the reader is encouraged to look at the general result in \cite{Sun/Zhang:2021}, and make their own comparisons.
\end{remark}

As immediate corollaries, we get the following results.

\begin{corollary}\label{Corollary admissible subcategories from Tc to Tbc with pre-approximability}
     Let $\mathsf{T}$ be a  triangulated category with a single compact generator $G$ such that $\HomT{G}{\Sigma^{n}G} = 0$ for all $n >> 0$. We equip it with the $\mathbb{N}$-compressed generating sequence (\Cref{Definition compressed metrics}) $\mathcal{G}$ given by $\mathcal{G}^n = \{\Sigma^{-n}G\}$ for all $n \in \mathbb{Z}$, see \Cref{Theorem Examples of Compressed metrics}. Then, for any admissible category $\mathsf{A}$ of $\mathsf{T}^c$, we have that $\operatorname{Coprod}(\mathsf{A})\cap{\mathsf{T}^b_c}$ (see \Cref{Notation from Neeman}) is a right admissible subcategory of $\mathsf{T}^b_c$ (\Cref{Definition closure of compacts}).
\end{corollary}
\begin{proof}
    This is immediate from \Cref{Theorem admissible categories and T^b_c}.
\end{proof}

\begin{corollary}\label{Corollary admissible subcategories from Tc to Tbc with co-approximability}
    Let $\mathsf{T}$ be a  triangulated category with a single compact generator $G$ such that $\HomT{G}{\Sigma^{n}G} = 0$ for all $n << 0$. We equip it with the $\mathbb{N}$-compressed generating sequence (\Cref{Definition compressed metrics}) $\mathcal{G}$ given by $\mathcal{G}^n = \{\Sigma^n G\}$ for all $n \in \mathbb{Z}$, see \Cref{Theorem Examples of Compressed metrics}. Then, for any admissible category $\mathsf{A}$ of $\mathsf{T}^c$, we have that $\operatorname{Coprod}(\mathsf{A})\cap{\mathsf{T}^b_c}$ (see \Cref{Notation from Neeman}) is a right admissible subcategory of $\mathsf{T}^b_c$ (\Cref{Definition closure of compacts}).
\end{corollary}
\begin{proof}
     This is immediate from \Cref{Theorem admissible categories and T^b_c}.
\end{proof}

\section{Examples of co-approximability}\label{Section examples of co-approximability}
In this section we will give some examples of weakly co-approximable and weakly co-quasiapproximable triangulated categories coming from algebraic geometry. The main source of examples for (weak) co-approximability is the homotopy category of injectives for a nice enough noetherian scheme algebra, or stack. As mentioned in the introduction
%\Kabeer{Be sure to mention this in the introduction then.}
, this will help us compute the closure of compacts for the homotopy category of injectives, which we will do in \Cref{Section Examples from Algebraic Geometry}.

We begin in a more general setting. Recall that a locally noetherian Grothendieck abelian category is a Grothendieck category $\mathcal{C}$, which has a set of noetherian object $\mathcal {D}$ which generate it, that is, every object is a quotient of coproduct of objects in $\mathcal{D}$. We will denote the full subcategory of noetherian objects by $\operatorname{noeth} \mathcal{C}$. The category we will be interested in is the homotopy category of injectives $\mathbf{K}(\operatorname{Inj} \mathcal{C})$. By \cite[Proposition 2.3]{Krause:2004}, $\mathbf{K}(\operatorname{Inj} \mathcal{C})$ is compactly generated, and the full subcategory of compacts is equivalent canonically to the bounded derived category of noetherian objects $\mathbf{D}^b(\operatorname{noeth} \mathcal{C})$.

\begin{notation}
    Let $\mathcal{C}$ be a locally noetherian Grothendieck abelian category. Then, 
    \begin{itemize}
        \item $(\mathbf{D}(\mathcal{C})^{\leq 0},\mathbf{D}(\mathcal{C})^{\geq 0})$ denotes the standard t-structure on $\mathbf{D}(\mathcal{C})$, where the truncation triangles are given by the soft or the canonical truncation of complexes.
        \item $(\mathbf{K}(\operatorname{Inj} \mathcal{C})^{\geq 0},\mathbf{K}(\operatorname{Inj} \mathcal{C})^{\leq 0})$ denotes the standard co-t-structure on $\mathbf{K}(\operatorname{Inj} \mathcal{C})$, where the truncation triangles are given by the hard or the brutal truncation of complexes.
        
    \end{itemize}
\end{notation}
For a noetherian scheme $X$, we will work with the locally noetherian Grothendieck abelian category of quasicoherent sheaves $\operatorname{Qcoh} X$.

We start with a couple of easy and well known lemmas.

\begin{lemma}\label{Lemma coaisle of derived category and "aisle" of homotopy category}
    Let $\mathcal{C}$ be a locally noetherian Grothendieck abelian category and we denote by $p : \mathbf{K}(\operatorname{Inj} \mathcal{C}) \to \mathbf{D}(\mathcal{C})$ the Verdier quotient map, see \cite[Proposition 3.6]{Krause:2004}. Then, this restricts to an equivalence $ p : \mathbf{K}(\operatorname{Inj} \mathcal{C})^{\geq 0} \to \mathbf{D}(\mathcal{C})^{\geq 0} $.
\end{lemma}
\begin{proof}
    This follows trivially as $\Hom{\mathbf{K}(\operatorname{Inj} \mathcal{C})}{E}{F} \cong \Hom{\mathbf{D}(\mathcal{C})}{E}{F}$ if $F$ is a bounded below complex of injective objects, and $E$ is an arbitrary complex of injective objects, see \cite[Tag 05TG]{StacksProject}.
\end{proof}
\begin{remark}
    In particular \Cref{Lemma coaisle of derived category and "aisle" of homotopy category} implies that for any $G \in \mathbf{D}^{b}(\operatorname{noeth} \mathcal{C})$, the subcategories $\overline{\langle G \rangle}^{[A,B]}_C$ (see \Cref{Notation from Neeman}) for any integers $A,B$ and positive integer $C$ remain the same viewed in either $\mathbf{K}(\operatorname{Inj} \mathcal{C})$ or $\mathbf{D}(\mathcal{C})$, that is they are equivalent under the functor $p$. We will confuse these equivalent categories freely throughout the rest of this document.
\end{remark}
\begin{lemma}\label{Lemma similar to Lemma 2.1 Krause:2004}\cite[Lemma 2.1]{Krause:2004}
    Let $\mathcal{C}$ be a locally noetherian Grothendieck abelian category. Let $F$ be a bounded below complex, and $I$ an arbitrary complex of injective sheaves. Then, $\Hom{\mathbf{K}(\mathcal{C})}{F}{I} \cong \Hom{\mathbf{K}(\operatorname{Inj} \mathcal{C})}{I_F}{I}$, where $I_F$ is an injective resolution of $F$.
\end{lemma}
\begin{proof}
     We complete the map that we get from the injective resolution to a triangle $F \to I_F \to \operatorname{Cone}(F \to I_F) \to \Sigma F$ in $\mathbf{K}(\mathcal{C})$. Note that $\operatorname{Cone}(F \to I_F)$ is an acyclic complex. So, $\Hom{}{\operatorname{Cone}(F \to I_F)}{\sigma^{\geq n}I} = 0$ by \cite[Tag 013R]{StacksProject} for any $n \in \mathbb{Z}$, where $\sigma^{\geq n}I$ denote the brutal truncation of $I$. But, as $F$ is a bounded below complex, there exists an integer $N$ such that $F^i = 0$ for all $i \leq N$. So,
     \[\Hom{}{\Sigma^n\operatorname{Cone}(F \to I_F)}{I} = \Hom{}{\operatorname{Cone}(F \to I_F)}{\sigma^{\geq N-2}I} = 0\]
     And so, $\Hom{\mathbf{K}(\mathcal{C})}{F}{I} \cong \Hom{\mathbf{K}(\operatorname{Inj}\mathcal{C})}{I_F}{I}$
\end{proof}

We now come to a result on weak co-approximability.

\begin{proposition}\label{Proposition weak co-approx for Kinj}
    Let $\mathcal{C}$ be a locally noetherian Grothendieck abelian category. Suppose there exists an object $\hat{G} \in \mathbf{D}^b(\operatorname{noeth}(\mathcal{C}))$ and a positive integer $N$ such that,
    \[\mathbf{D}(\mathcal{C})^{\geq 0} = \overline{\langle \hat{G} \rangle}^{[-N,N]} \star \mathbf{D}(\mathcal{C})^{\geq 1}\] see \Cref{Notation from Neeman}.
    Then, 
    \begin{enumerate}[label=(\roman*)]
        \item $\hat{G}$ is a compact generator for $\mathbf{K}(\operatorname{Inj} \mathcal{C})$.
        \item $\mathbf{K}(\operatorname{Inj} \mathcal{C})$ is weakly co-approximable.
        \item The co-t-structure $(\mathbf{K}(\operatorname{Inj} \mathcal{C})^{\geq 0},\mathbf{K}(\operatorname{Inj} \mathcal{C})^{\leq 0})$ lies in the preferred quasiequivalence class, see \Cref{Definition preferred equivalence class}.
    \end{enumerate}   
\end{proposition}
\begin{proof} 
    We first show that $\hat{G}$ is a compact generator for $\mathbf{K}(\operatorname{Inj} \mathcal{C})$.

     Let $F \in \mathbf{K}(\operatorname{Inj} \mathcal{C})^{\geq 0} \cong \mathbf{D}(\mathcal{C})^{\geq 0}$, see \Cref{Lemma coaisle of derived category and "aisle" of homotopy category}. Then, by induction we will construct a sequence $F_1 \to F_2 \to F_3 \to F_4 \to \cdots $ mapping to $F$ such that for some integer $N \geq 0$,
    \begin{itemize}
        \item $F_i \in \overline{\langle \hat{G} \rangle}^{[-N,\infty)}$ for all $i \geq 1$.
        \item $D_{n} \colonequals \operatorname{Cone}(F_i \to F) \in \mathbf{K}(\operatorname{Inj} \mathcal{C})^{\geq i}$ for all $i \geq 1$.
    \end{itemize}
    We get $F_1$ by \Cref{Theorem on approximating on J-2 schemes}. Now, suppose we have the sequence $F_1 \to \cdots \to F_n$. By the induction hypothesis, we know that $D_n \colonequals \operatorname{Cone}(F_n \to F) \in \mathbf{K}(\operatorname{Inj} \mathcal{C})^{\geq n} \cong \mathbf{D}(\mathcal{C})^{\geq n}$. Applying \Cref{Theorem on approximating on J-2 schemes} to $\Sigma^n D_n$, we get a triangle $\widetilde{F} \to D_n \to D_{n+1} \to \Sigma \widetilde{F}$ such that $\widetilde{F} \in \overline{\langle \hat{G} \rangle}^{[-N+n,N+n]}$ and $D_{n+1} \in \mathbf{K}(\operatorname{Inj} \mathcal{C})^{\geq n + 1}$. Applying the octahedral axiom to $F \to D_n \to D_{n+1}$, we get,
    \[\begin{tikzcd}
	   F_n & F_{n+1} & \widetilde{F} \\
	   F_n & F & D_n \\
	     & D_{n+1} & D_{n+1}
	   \arrow[from=1-1, to=1-2]
	   \arrow[from=1-1, to=2-1]
          \arrow[from=1-2, to=1-3]
	   \arrow[from=1-2, to=2-2]
          \arrow[from=1-3, to=2-3]
	   \arrow[from=2-1, to=2-2]
	   \arrow[from=2-2, to=2-3]
	   \arrow[from=2-2, to=3-2]
	   \arrow[from=2-3, to=3-3]
	   \arrow[from=3-2, to=3-3]
    \end{tikzcd}\]
    The middle column gives us the required triangle, as, 
    \[F_{n+1} \in \overline{\langle \hat{G} \rangle}^{[-N,\infty)} \star \overline{\langle \hat{G} \rangle}^{[-N+n,N+n]} \subseteq  \overline{\langle \hat{G} \rangle}^{[-N,\infty)}\]
    see \Cref{Notation from Neeman}. This gives us the required sequence mapping to $F$ by induction.
    
    Consider the non-canonical map $\hocolim F_i \to F$. For any object $H \in \mathbf{D}^{b}(\operatorname{noeth} \mathcal{C})$, we have the map given by the composite 
    \[\colim \Hom{}{H}{F_i} \overset{\thicksim}{\longrightarrow}\Hom{}{H}{\hocolim F_i} \to \Hom{}{H}{F}\] 
    where the first map is an isomorphism by \cite[Lemma 2.8]{Neeman:1996}. In the triangle $\Sigma^{-1}D_i \to F_i \to F \to D_i$, with $D_i \in \mathbf{K}(\operatorname{Inj} \mathcal{C})^{\geq i}$ for all $i\geq 1$, we get that,
    \[\Hom{}{H}{\Sigma^{-1} D_i} = \Hom{}{H}{D_i} = 0 \text{ for all } i>>0 \] by \Cref{Lemma similar to Lemma 2.1 Krause:2004}. So, by applying the functor $\Hom{}{H}{-}$ to the triangle, we get that $ \Hom{}{H}{F_i} \cong \Hom{}{H}{F}$ for $i >> 0$. And so, we get that $\colim \Hom{}{H}{F_i} \cong \Hom{}{H}{F}$ for all $H \in \mathbf{D}^{b}(\operatorname{noeth} \mathcal{C})$. Hence, as $\mathbf{D}^{b}(\operatorname{noeth} \mathcal{C})$ compactly generates $\mathbf{K}(\operatorname{Inj} \mathcal{C})$, by \cite[Lemma 2.8]{Neeman:1996} we get that $\hocolim F_i \cong F \in \overline{\langle \hat{G} \rangle}^{[-N-1,\infty)}$. So, $\bigcup_{n \in \mathbb{Z}}\mathbf{K}(\operatorname{Inj} \mathcal{C})^{\geq n} \subseteq \overline{\langle \hat{G} \rangle}$.

    Now, let $F \in \mathbf{K}(\operatorname{Inj} \mathcal{C})$. Let us choose a representative complex of injectives for it, which we will also denote $F$. Let its brutal truncations be denoted by $\sigma^{\geq -i}F$, where we have $\sigma^{\geq -i}F \in \mathbf{K}(\operatorname{Inj} \mathcal{C})^{\geq -i}$. Then, it is easy to see that $\hocolim \sigma^{\geq -i} F = F$, which proves that $\mathbf{K}(\operatorname{Inj} \mathcal{C}) = \overline{\langle \hat{G} \rangle}$, which proves $(i)$.

    Now we prove $(ii)$.
    For any object $H \in \mathbf{D}^{b}(\operatorname{noeth} \mathcal{C})$, there exists some $N_H > 0$ such that $\Sigma^{-N_H}H \in \mathbf{K}(\operatorname{Inj} \mathcal{C})^{\geq 0}$ and $\HomT{\Sigma^{i}H}{\mathbf{K}(\operatorname{Inj} \mathcal{C})^{\geq 0}} = 0$ for all $i \geq N_H$. In particular, this is true for $\hat{G} \in \mathbf{D}^{b}(\operatorname{noeth} \mathcal{C})$. Finally, let $F \in \mathbf{K}(\operatorname{Inj} \mathcal{C})^{\geq 0} \cong \mathbf{D}(\mathcal{C})^{\geq 0}$, see \Cref{Lemma coaisle of derived category and "aisle" of homotopy category}. Then, by our hypothesis, there exists a triangle $E \to F \to D \to \Sigma E$ with $E \in \overline{\langle \hat{G} \rangle}^{[N,N]}$ and $D \in \mathbf{D}(\mathcal{C})^{\geq 1} \cong \mathbf{K}(\operatorname{Inj} \mathcal{C})^{\geq 1}$. So, we have proven the weak co-approximability of $\mathbf{K}(\operatorname{Inj} \mathcal{C})$ with the standard co-t-structure, and the integer $\max(N_{\hat{G}},N)$.

    Finally, $(iii)$ immediately follows from \Cref{Lemma Coapprox implies preferred equivaence class} from the proof of $(ii)$.
\end{proof}

We now move to the setting of algebraic geometry.
We begin by recalling certain notions related to the regular locus of a scheme to state some of the later results.
\begin{definition}
    Let $X$ be a noetherian scheme. Recall that the regular locus $\operatorname{reg} X$ is the collection of points $p \in X$ such that the stalk $\mathcal{O}_{X,p}$ is a regular local ring. Then,
    \begin{itemize}
        \item $X$ is J-0 if there is a non-empty open subset contained in $\operatorname{reg} X$.
        \item $X$ is J-1 if $\operatorname{reg} X$ is open. Note that this does not imply in general that $X$ is J-0 as $\operatorname{reg} X$ can be empty. But, if $X$ is reduced, then it being J-1 implies being J-0 as the regular locus of a noetherian reduced scheme is non-empty. The non-emptiness of the regular locus follows from the fact that for a reduced scheme, the local rings at the generic points of each irreducible component is regular.
        \item $X$ is J-2 if for every morphism $f : Y \to X$ of finite type, the regular locus $\operatorname{reg} Y$ is open in $Y$. 
    \end{itemize}
    
\end{definition}

We now prove the weak approximability of $\mathbf{K}(\operatorname{Inj} X)$ under a mild hypothesis. For that, the main theorem we need to prove in preparation is the following. We state it here, and then prove it using a sequence of lemmas.

\begin{theorem}\label{Lemma on approximating on J-2 schemes}
    Let $X$ be a finite dimensional noetherian scheme such that each integral closed subscheme is J-0 and let $\mathsf{T}$ be any triangulated subcategory of $\mathbf{D}(\operatorname{Qcoh} X)$ . Then, there is an object $\hat{G}$ in $ \mathbf{D}^{b}(\operatorname{coh} X)$ such that, 
    \begin{enumerate}
        \item There exists an integer $N \geq 0$ with $\mathsf{T} \cap \operatorname{Qcoh} X \subseteq \overline{\langle \hat{G} \rangle}^{[-N,N]}$.
        \item There exist integers $N_{p,q} \geq 0$ for each pair of integers $p \leq q $ such that 
        \[\mathsf{T} \cap \mathbf{D}(\operatorname{Qcoh} X)^{\geq p} \cap \mathbf{D}(\operatorname{Qcoh} X)^{\leq q} \subseteq \overline{\langle \hat{G} \rangle}^{[-N_{p,q},N_{p,q}]}\]
    
    \end{enumerate}
    see \Cref{Notation from Neeman}.
\end{theorem}
\begin{remark}\label{Remark on approximating on J-2 schemes}
    It is clear that it is enough to show \Cref{Lemma on approximating on J-2 schemes} for $\mathsf{T} = \mathbf{D}(\operatorname{Qcoh} X)$. In fact, we are only interested in the statement for $\mathsf{T} = \mathbf{D}(\operatorname{Qcoh} X)$, and we introduce $\mathsf{T}$ just as a placeholder as it makes it easier to state and use the aforementioned sequence of lemmas.

    Further, note that \Cref{Lemma on approximating on J-2 schemes}$(1) \implies (2)$. This can be shown by an induction on $(q-p)$ as follows. Suppose $F \in \mathbf{D}(\operatorname{Qcoh} X)^{\geq p} \cap \mathbf{D}(\operatorname{Qcoh} X)^{\leq q}$ for integers $p$ and $q$ such that $q-p \geq 1$. Then, we have the triangle $\tau^{\leq q-1}F \to F \to \tau^{\geq q}F \to \Sigma \tau^{\leq q-1}F$ with $\tau^{\leq q-1}F \in \mathbf{D}(\operatorname{Qcoh} X)^{\geq p} \cap \mathbf{D}(\operatorname{Qcoh} X)^{\leq q-1}$ and $\tau^{\geq q}F \in \mathbf{D}(\operatorname{Qcoh} X)^{\geq q} \cap \mathbf{D}(\operatorname{Qcoh} X)^{\leq q}$ given by the canonical truncation. But, 
    \[\mathbf{D}(\operatorname{Qcoh} X)^{\geq q} \cap \mathbf{D}(\operatorname{Qcoh} X)^{\leq q} = \Sigma^{-q} (\mathbf{D}(\operatorname{Qcoh} X)^{\geq 0} \cap \mathbf{D}(\operatorname{Qcoh} X)^{\leq 0}\]
    so this shows the inductive step.
\end{remark}

\begin{lemma}\label{Lemma 1 on approximating on J-2 schemes}
    Let $X$ be noetherian scheme and $i : Z \to X$ a closed immersion such that the conclusion of \Cref{Lemma on approximating on J-2 schemes} holds for the scheme $Z$. Then, the conclusion of \Cref{Lemma on approximating on J-2 schemes} holds for $X$ with $\mathsf{T} = \mathbf{D}_{Z}(\operatorname{Qcoh} X)$, which is the thick subcategory of $\mathbf{D}(\operatorname{Qcoh} X)$ consisting of complexes with cohomology supported on $Z$.
\end{lemma}
\begin{proof}
    By \Cref{Remark on approximating on J-2 schemes}, it is enough to show that \Cref{Lemma on approximating on J-2 schemes}(1) holds for $X$ with $\mathsf{T} = \mathbf{D}_{Z}(\operatorname{Qcoh} X)$. Let $\mathcal{I}$ be the sheaf of ideals corresponding to the closed immersion $i : Z \to X$. Let $F \in \operatorname{Qcoh}_{Z}(X) = \mathbf{D}_{Z}(\operatorname{Qcoh} X) \cap \operatorname{Qcoh} X$. For each $n \geq 0$, let $F_n$ denote the subsheaf of $F$ which is annihilated by $\mathcal{I}^n$. That is, we define the sections of $F_n$ for each open subset $U \subseteq X$ by, $F_n(U) \colonequals \{s \in F(U) : \mathcal{I}(U)^n \cdot s = 0\}$. Then, we get an exhaustive filtration of $F$ as follows,
     \[0 = F_0 \subseteq F_1 \subseteq F_2 \subseteq F_3 \subseteq \cdots \subseteq F\]
     So, $F = \colim F_n$ in $\operatorname{Qcoh} X$, which gives us that in the derived category $F \cong \hocolim F_n$ by \cite[Remark 2.2]{Bkstedt/Neeman:1993}.

     Note that each successive quotient, $F_{n+1}/F_n$ is annihilated by $\mathcal{I}$, and so is an $\mathcal{O}_X/\mathcal{I}$-module, that is, $F_{i+1}/F_i \in i_*(\operatorname{Qcoh} Z) $. Now, from the short exact sequences $0 \to F_n \to F_{n+1} \to F_{n+1}/F_n \to 0$, we get triangles $F_n \to F_{n+1} \to F_{n+1}/F_n \to \Sigma F_n$ for each $n \geq 0$. This gives us that $F_{n+1} \in F_n \star (F_{n+1}/F_n)$ for each $n \geq 0$. 
     
     As $F \cong \hocolim F_n$, we get that $F \in \overline{\langle \{F_{n+1}/F_n\}_{n \geq 0}}\rangle^{[-1,1]}$, see \Cref{Notation from Neeman}. But, note that $\overline{\langle\{F_{n+1}/F_n\}_{n \geq 0}}\rangle^{[-1,1]} \subseteq \overline{\langle i_*\mathbf{D}(\operatorname{Qcoh} Z})\rangle^{[-1,1]}$. By the hypothesis there exists an object $\widetilde{G} \in \mathbf{D}^b({\operatorname{coh}}(Z))$ and $N \geq 1$ such that,
     $\operatorname{Qcoh} Z \subseteq \overline{\langle \widetilde{G}\rangle}^{[-N,N]}$. Define $\hat{G} \colonequals i_*(\widetilde{G})$. Then, 
     \[\operatorname{Qcoh}_{Z}(X) \subseteq \overline{\langle i_*(\operatorname{Qcoh} Z\rangle}^{[-1,1]} \subseteq i_*\Big( \overline{\langle \widetilde{G}\rangle}^{[-N-1,N+1]} \Big) \subseteq \overline{\langle \hat{G} \rangle}^{[-N-1,N+1]} \]
     which gives us the required result.
\end{proof}

\begin{lemma}\label{Lemma 2 on approximating on J-2 schemes}
    Let $X$ be a noetherian scheme with closed subschemes $Z$ and $Z^{\prime}$ such that $X = Z \cup Z^{\prime}$ topologically. If the conclusion of \Cref{Lemma on approximating on J-2 schemes} holds for $Z$ and $Z^{\prime}$, then it holds for $X$.
\end{lemma}
\begin{proof}
    First of all, we immediately get that the result holds for $X$ with $\mathsf{T} = \mathbf{D}_{Z}(\operatorname{Qcoh} X)$ and $\mathsf{T} = \mathbf{D}_{Z^\prime}(\operatorname{Qcoh} X)$ by \Cref{Lemma 1 on approximating on J-2 schemes}.
    By \Cref{Remark on approximating on J-2 schemes}, it is enough to show that \Cref{Lemma on approximating on J-2 schemes}(1) holds for $X$ with $\mathsf{T} = \mathbf{D}(\operatorname{Qcoh} X)$. Let $U = X \setminus Z$ and $U' = X \setminus Z^{\prime}$ with corresponding open immersions $j : U \to X$ and $j' : U' \to X$.  
    
    Let $F \in \operatorname{Qcoh} X$. Consider the triangle $F' \to F \to \mathbb{R}j_*j^*F \to \Sigma F'$, where the second map is the unit of adjunction. Note that when restricted to $U$, we get that the unit of adjunction is an isomorphism, and so $j^*(F') \cong 0$. And so, $F' \in \mathbf{D}_{Z}(\operatorname{Qcoh} X) $. By \cite[Proposition 3.9.2]{Lipman:2009}, $\mathbb{R}j_*j^*F \in \mathbf{D}(\operatorname{Qcoh} X)^{\geq 0} \cap \mathbf{D}(\operatorname{Qcoh} X)^{\leq t-1}$ for some positive integer $t$. So, as $F' \in (\Sigma^{-1} \mathbb{R}j_*j^*F) \star F $, we get that, 
    \[F' \in \mathbf{D}_{Z}(\operatorname{Qcoh} X)^{\geq 0} \cap \mathbf{D}_{Z}(\operatorname{Qcoh} X)^{\leq t}\]
    This in turn implies that,
    \[F \in \big(\mathbf{D}_{Z}(\operatorname{Qcoh} X)^{\geq 0} \cap \mathbf{D}_{Z}(\operatorname{Qcoh} X)^{\leq t}\big) \star \mathbb{R}j_*\operatorname{Qcoh} U\]
    for some positive integer $t$. 
    
    Further, note that $j'^{,*}\mathbb{R}j_*\mathbf{D}(\operatorname{Qcoh} U) = 0$ from a simple application of flat base change. This immediately implies that $\mathbb{R}j_*\operatorname{Qcoh} U \subseteq \mathbf{D}_{Z^\prime}(\operatorname{Qcoh} X)^{\geq 0} \cap \mathbf{D}_{Z^\prime}(\operatorname{Qcoh} X)^{\leq t-1}$. So, we get that,
    \[\operatorname{Qcoh} X = \big(\mathbf{D}_{Z}(\operatorname{Qcoh} X)^{\geq 0} \cap \mathbf{D}_{Z}(\operatorname{Qcoh} X)^{\leq t}\big) \star \big(\mathbf{D}_{Z^\prime}(\operatorname{Qcoh} X)^{\geq 0} \cap \mathbf{D}_{Z^\prime}(\operatorname{Qcoh} X)^{\leq t-1}\big)\]
    which immediately gives us the required result as the conclusion of \Cref{Lemma on approximating on J-2 schemes} holds for $X$ with $\mathsf{T} = \mathbf{D}_{Z}(\operatorname{Qcoh} X)$ and $\mathsf{T} = \mathbf{D}_{Z^\prime}(\operatorname{Qcoh} X)$.
\end{proof}
Now we can prove \Cref{Lemma on approximating on J-2 schemes}.
\begin{proof}[Proof of \Cref{Lemma on approximating on J-2 schemes}]
    Let $X$ be a noetherian scheme such that each integral closed subscheme is J-0.
    We start by observing some reductions we can do. First of all, by \Cref{Remark on approximating on J-2 schemes}, it is enough to prove \Cref{Lemma on approximating on J-2 schemes}(1) for $\mathsf{T} = \mathbf{D}(\operatorname{Qcoh} X)$. Further, if the result holds for all the irreducible components of $X$, then it holds for $X$ using \Cref{Lemma 2 on approximating on J-2 schemes}. Finally, let $X_{\operatorname{red}}$ be the reduced scheme corresponding to $X$ with the corresponding closed immersion $i : X_{\operatorname{red}} \to X$. If we know that the conclusion of \Cref{Lemma on approximating on J-2 schemes} holds for $X_{\operatorname{red}}$, then it also holds for $X$ by \Cref{Lemma 1 on approximating on J-2 schemes}. So we can assume that the scheme is reduced and irreducible, that is, the scheme is integral.

    We now proceed by induction on the dimension of the scheme $X$. For the base case, let $X$ be a integral scheme such that $\dim(X)=0$. Then $X$ is just $\operatorname{Spec}(k)$ for a field $k$. Suppose $F \in \operatorname{Qcoh} X\cong \operatorname{Mod} k$. As $F$ is a vector space over $k$, $F \in \overline{\langle k\rangle}^{[0,0]}$ (\Cref{Notation from Neeman}), and hence we get the required result in this case with $\hat{G} \colonequals k$.

    Now, we assume the result is known for dimension smaller than dim$(X)$.
    As the scheme is integral, there exists a non-empty affine open set $U = \operatorname{Spec}(R) \subseteq \operatorname{reg} X$ by hypothesis. Let the corresponding open immersion be $j : U \to X$. Let $F \in \operatorname{Qcoh} X$. Then, the restriction to the open set $U$, $j^*F \in \operatorname{Qcoh} U$. As $U = \operatorname{Spec}(R) \subseteq \operatorname{reg} X$, the ring $R$ is regular. So $R$ is a regular ring of Krull dimensional less than or equal to $\dim(X) = n$, and hence the global dimension is less than or equal to $n$. So, $j^*F$ must have a projective resolution of length less than or equal to $n$. This gives us that $j^*F \in \overline{\langle R \rangle}^{[- n, n]} = \overline{\langle j^*\mathcal{O}_X \rangle}^{[- n, n]}$.
    
    Consider the triangle $F' \to F \to \mathbb{R}j_*j^*F \to \Sigma F'$, where the second map is the unit of adjunction. Note that when restricted to $U$, we get that the unit of adjunction is an isomorphism, and so $j^*(F') \cong 0$. Let $Z=X-U$. Then, $F' \in \mathbf{D}_{Z}(\operatorname{Qcoh} X)$. By \cite[Proposition 3.9.2]{Lipman:2009}, there exists $t \geq 0$  such that 
    \[\mathbb{R}j_*j^*F \in \mathbf{D}(\operatorname{Qcoh} X)^{\geq 0} \cap \mathbf{D}(\operatorname{Qcoh} X)^{\leq t-1}\]
    So, as $F' \in (\Sigma^{-1} \mathbb{R}j_*j^*F) \star F $ we get that, 
    \[F' \in \mathbf{D}_{Z}(\operatorname{Qcoh} X)^{\geq 0} \cap \mathbf{D}_{Z}(\operatorname{Qcoh} X)^{\leq t} \]
    This gives us that,
    \[F \in \big(\mathbf{D}_{Z}(\operatorname{Qcoh} X)^{\geq 0} \cap \mathbf{D}_{Z}(\operatorname{Qcoh} X)^{\leq t} \big)\star \overline{\langle \mathbb{R}j_*j^*\mathcal{O}_X \rangle}^{[- n, n]}\]
    as $F \in F' \star \mathbb{R}j_*j^*F$, and as $\mathbb{R}j_*j^*F \in \mathbb{R}j_*\big(\overline{\langle j^*\mathcal{O}_X \rangle}^{[- n, n]}\big) \subseteq \overline{\langle \mathbb{R}j_*j^*\mathcal{O}_X \rangle}^{[- n, n]}$ by the previous paragraph. Note that the integer $t$ can be chosen independent of the choice of $F$ in $\operatorname{Qcoh} X$.

    Consider the triangle $Q \to \mathcal{O}_X \to \mathbb{R}j_*j^*\mathcal{O}_X \to \Sigma Q$ in $\mathbf{D}(\operatorname{Qcoh} X)$ coming from the unit of adjunction. Note that when restricted to $U$, we get that the unit of adjunction is an isomorphism, and so $j^*(Q) \cong 0$. Therefore, $\Sigma Q \in \mathbf{D}_{Z}(\operatorname{Qcoh} X)^{\geq -1} \cap \mathbf{D}_{Z}(\operatorname{Qcoh} X)^{\leq t}$. Further, $\mathcal{O}_X \in \langle G \rangle^{[-C,C]}$ for a compact generator $G \in \mathbf{D}^{\operatorname{perf}}(X)$ and some positive integer $C$. So, \[\mathbb{R}j_*j^*\mathcal{O}_X \in \langle G \rangle^{[-C,C]} \star \big(\mathbf{D}_{Z}(\operatorname{Qcoh} X)^{\geq -1} \cap \mathbf{D}_{Z}(\operatorname{Qcoh} X)^{\leq t}\big)\] 

    Now, let $i : Z \to X$ be the closed immersion where $Z$ is given the reduced induced closed subscheme structure. First of all, note that as $\dim(Z) < \dim(X)$, $Z$ satisfies the conclusion of \Cref{Lemma on approximating on J-2 schemes}. We define $\hat{G} = i_* \hat{H} \oplus G$, where $\hat{H} \in \mathbf{D}^b({\operatorname{coh}}(Z))$ is the object coming from $Z$ satisfying \Cref{Lemma on approximating on J-2 schemes}. As $F \in \big(\mathbf{D}_{Z}(\operatorname{Qcoh} X)^{\geq 0} \cap \mathbf{D}_{Z}(\operatorname{Qcoh} X)^{\leq t} \big)\star \overline{\langle \mathbb{R}j_*j^*\mathcal{O}_X \rangle}^{[- n, n]}$ and $\mathbb{R}j_*j^*\mathcal{O}_X \in \overline{\langle G \rangle}^{[-C,C]} \star \big(\mathbf{D}_{Z}(\operatorname{Qcoh} X)^{\geq -1} \cap \mathbf{D}_{Z}(\operatorname{Qcoh} X)^{\leq t}\big)$ from the above paragraph, we would be done if we can show that for all $p \leq q$, there exists an integer $L_{p,q}$ such that,
    \[\mathbf{D}_{Z}(\operatorname{Qcoh} X)^{\geq p} \cap \mathbf{D}_{Z}(\operatorname{Qcoh} X)^{\leq q} \subseteq \overline{\langle \hat{G} \rangle}^{[-L_{p,q},L_{p,q}]}\]
    But, as the conclusion of \Cref{Lemma on approximating on J-2 schemes} holds for $Z$, we get this from \Cref{Lemma 1 on approximating on J-2 schemes}, completing the proof.
\end{proof}

\begin{corollary}\label{Theorem on approximating on J-2 schemes}
    Let $X$ be a noetherian finite-dimensional scheme such that each integral closed subscheme is J-0. Then, there exists an object $\hat{G} \in \mathbf{D}^{b}(\operatorname{coh} X)$ and a positive integer $N$ such that $\mathbf{D}(\operatorname{Qcoh} X)^{\geq 0} = \overline{\langle \hat{G} \rangle}^{[-N,N]} \star \mathbf{D}(\operatorname{Qcoh} X)^{\geq 1}$, see \Cref{Notation from Neeman}.
\end{corollary}
\begin{proof}
    Let $F \in \mathbf{D}(\operatorname{Qcoh} X)^{\geq 0}$. Then, we have the triangle $F^{\leq 0} \to F \to F^{\geq 1} \to \Sigma F^{\leq 0}$ we get from the standard t-structure on $\mathbf{D}(\operatorname{Qcoh} X)$. So, we have that $F^{\leq 0} \in \operatorname{Qcoh} X$ and $F^{\geq 1} \in \mathbf{D}(\operatorname{Qcoh} X)^{\geq 1}$. But, by \Cref{Lemma on approximating on J-2 schemes}, there exists $N \geq 0$ such that $\operatorname{Qcoh} X \subseteq \overline{\langle \hat{G} \rangle}^{[-N,N]}$. So, $F \in \overline{\langle \hat{G} \rangle}^{[-N,N]} \star \mathbf{D}(\operatorname{Qcoh} X)^{\geq 1}$.
\end{proof}

The following result is known, see for example \cite[Theorem 1.1]{Lank/Dey:2024} and \cite[Theorem 4.15]{Elagin/Lunts/Schnurer:2020}, but we give an alternate proof using \Cref{Theorem on approximating on J-2 schemes}.

\begin{theorem}\label{Theorem Compact Generator for Kinj}
    Let $X$ be a noetherian finite-dimensional scheme such that each integral closed subscheme is J-0. Then, $\mathbf{K}(\operatorname{Inj} X)$ has a single compact generator. In fact, the object $\hat{G}$ of \Cref{Theorem on approximating on J-2 schemes} is a compact generator of $\mathbf{K}(\operatorname{Inj} X)$. 
    
    In particular, $\hat{G}$ is a classical generator for $\mathbf{D}^b({\operatorname{coh} X})$.
\end{theorem}
\begin{proof}
    Immediate from \Cref{Proposition weak co-approx for Kinj}$(i)$ using \Cref{Theorem on approximating on J-2 schemes}.
\end{proof}

\begin{theorem}\label{Theorem J-2 implies weak co-approximablility for schemes}
    Let $X$ be a noetherian finite-dimensional scheme such that each integral closed subscheme is J-0. Then $\mathbf{K}(\operatorname{Inj} X)$ is a weakly co-approximable triangulated category, see \Cref{Definition Co-approx}. Further, the standard co-t-structure lies in the preferred equivalence class, see \Cref{Definition preferred equivalence class}. 
    
    Note that the hypothesis holds in particular if $X$ is J-2.
\end{theorem}
\begin{proof}
    Immediate from \Cref{Proposition weak co-approx for Kinj}$(ii),(iii)$ using \Cref{Theorem on approximating on J-2 schemes}.
\end{proof}

We also have a noncommutative version of the above results.

\begin{theorem}\label{Theorem J-2 implies weak co-approximablility for algebras}
    Let $X$ be a noetherian finite-dimensional J-2 scheme, and $\mathcal{A}$ any coherent $\mathcal{O}_X$-algebra. Then $\mathbf{K}(\operatorname{Inj} \mathcal{A})$ is a weakly co-approximable triangulated category, see \Cref{Definition Co-approx}. Further, the standard co-t-structure lies in the preferred equivalence class, see \Cref{Definition preferred equivalence class}. 
\end{theorem}
\begin{proof}
    Immediate from \Cref{Proposition weak co-approx for Kinj}$(ii),(iii)$ using \Cref{Theorem on approximating on J-2 algebra}.
\end{proof}

We now prove the weak co-quasiapproximability for the mock homotopy category of projectives $\mathbf{K}_m(\operatorname{Proj} X)$. We will follow the notational conventions from \cite{Murfet:2008}.

\begin{proposition}\label{Proposition Kmproj is Co-approx}
    Let $X$ be a noetherian, separated, finite dimensional scheme such that every integral closed subscheme is J-0. 
    Then, the mock homotopy category of projectives $\mathbf{K}_m(\operatorname{Proj} X)$ is weakly co-quasiapproximable, see \Cref{Definition Co-approx}. 
\end{proposition}
\begin{proof}
    We will be working with the co-t-structure $(\mathsf{U},\mathsf{V})$ compactly generated by the set $U_{\lambda}(\operatorname{coh} X)^{\circ}$, where $U_{\lambda}(-)^{\circ}: \mathbf{D}^b(\operatorname{coh} X)^{\operatorname{op}} \xrightarrow{\sim} \mathbf{K}_m(\operatorname{Proj} X)^c $, see \cite[Theorem 7.4]{Murfet:2008}. By \cite[Lemma 7.23]{ManaliRahul:2025}, the closure of compacts with respect to this co-t-structure is given by $\mathbf{K}_m(\operatorname{Proj}\text{-}X)^+_c = U_\lambda(D^{-}(\operatorname{coh} X)^{\circ}$, see \Cref{Definition closure of compacts} and \Cref{Convention closure of compacts for approx/co-approx}.

    Note that $\mathbf{K}_m(\operatorname{Proj}\text{-}X)^+_c \cap \mathsf{U}  \subseteq U_\lambda(\mathbf{D}(\operatorname{coh} X)^{\leq 1})^{\circ}$ by \cite[Theorem 2.3.4]{Bondarko:2022} and \cite[Proposition 1.9]{Neeman:2021a}. But, for any $F \in \mathbf{D}(\operatorname{coh} X)^{\leq 1}$, we have the truncation triangle $F^{\leq -1} \to F \to F^{\geq 0} \to \Sigma F^{\leq -1}$ in which $F^{\leq -1} \in \mathbf{D}(\operatorname{coh} X)^{\leq -1}$ and $F^{\geq 0} \in \mathbf{D}(\operatorname{coh} X)^{\leq 1} \cap \mathbf{D}(\operatorname{coh} X)^{\geq 0} \subseteq \overline{\langle G \rangle}^{[-N,N]}$ for some object $G \in \mathbf{D}^b(\operatorname{coh} X)$ and some integer $N \geq 0$ by \Cref{Lemma on approximating on J-2 schemes}.

    So, applying $U_{\lambda}(-)^{\circ}$ to the above triangle gives us the required approximating triangle. As the rest of the axioms of weak co-quasiapproximability are easily verified, we are done with the proof.
\end{proof}

\section{Examples from algebraic geometry}\label{Section Examples from Algebraic Geometry}

In this section, we apply the results in to categories coming from algebraic geometry. To apply the results, we first need to compute the closure of compacts in these situations.
\subsection*{Closure of compacts}
In this section, we will be working with compactly generated triangulated categories with a single compact generator $G$. We will equip the category with one of the two metrics of \Cref{Example generating sequence approx/co-approx}. Throughout the section after \Cref{Lemma closure of compacts for Co-approx}, we will only be considering the closure of compacts (\Cref{Definition closure of compacts}, \Cref{Convention closure of compacts for approx/co-approx}) with respect to a metric coming from a t-structure or a co-t-structure in the preferred quasiequivalence class, see \Cref{Definition preferred equivalence class}.  

We begin with the cases where we work with the generating sequence $\mathcal{G}$ given by $\mathcal{G}^i = \{\Sigma^i G \}$, see \Cref{Definition generating sequence}. The following lemma helps us in computing the closure of compacts.

\begin{lemma}\label{Lemma closure of compacts for Co-approx}
    Let $\mathsf{T}$ be a compactly generated triangulated category with a single compact generator $G$ and a co-t-structure $(\mathsf{U},\mathsf{V})$ on $\mathsf{T}$ such that $\Hom{}{\Sigma^iG}{\mathsf{U}} = 0$ for $i >> 0$. Then,
    \[\mathsf{T}^+_c = \{\hocolim E_i : E_i \in \mathsf{T}^c, \{E_i \to E_{i+1}\}_{i \geq 1} \text{ such that }\operatorname{Cone}(E_i \to E_{i+1}) \in \Sigma^{-i-1}\mathsf{U}\}\]
    where the closure of compacts is with respect to the metric $\{\Sigma^{-n}\mathsf{U}\}$, see \Cref{Convention closure of compacts for approx/co-approx}.
    
    In particular, the closure of compacts with respect to any other co-t-structure $(\mathsf{U}',\mathsf{V}')$ such that $\Sigma^{-N} \mathsf{U} \cap \mathsf{T}^c \subseteq \mathsf{U}'\cap \mathsf{T}^c \subseteq \Sigma^N \mathsf{U}\cap \mathsf{T}^c$ for some $N \geq 0$ is the same.  
\end{lemma}

\begin{proof}
    This is immediate from \Cref{Lemma approximating sequence for closure of compacts}.
\end{proof}

\begin{proposition}\label{Proposition closure of compacts for Kinj}
    Let $\mathcal{C}$ be a locally noetherian Grothendieck abelian category such that $\mathbf{K}(\operatorname{Inj} \mathcal{C})$ is weakly co-approximable and the standard co-t-structure lies in the preferred quasiequivalence class, see \Cref{Definition preferred equivalence class}. Recall that the full subcategory of compact objects $\mathbf{K}(\operatorname{Inj} \mathcal{C})^c \cong  \mathbf{D}^b(\operatorname{noeth}(\mathcal{C}))$ by \cite[Proposition 2.3]{Krause:2004}. Then,
    \begin{itemize}
        \item $\mathbf{K}(\operatorname{Inj} \mathcal{C})^+_c = \mathbf{D}^{+}(\operatorname{noeth}(X))$
        \item $\mathbf{K}(\operatorname{Inj} \mathcal{C})^- = \bigcup_{n \in \mathbb{Z}}\mathbf{K}(\operatorname{Inj} \mathcal{C})^{\leq n}$
        \item $\mathbf{K}(\operatorname{Inj} \mathcal{C})^b_c = \mathbf{D}^{b}(\operatorname{noeth} \mathcal{C}) \cap \mathbf{K}^{b}(\operatorname{Inj} \mathcal{C})$
    \end{itemize}
    with notation as in \Cref{Convention closure of compacts for approx/co-approx}.
\end{proposition}
\begin{proof}
    We can compute the closure of compacts \Cref{Lemma closure of compacts for Co-approx} with respect to the standard co-t-structure as it lies in the preferred equivalence class.

    Let $F \in \mathbf{K}(\operatorname{Inj} \mathcal{C})^+_c$. Then, by definition for each integer $n$ there exists a triangle $E_n \to F \to D_n \to \Sigma E_n$ with $E_n \in \mathbf{D}^{b}(\operatorname{noeth} \mathcal{C})$ and $\mathbf{K}(\operatorname{Inj} \mathcal{C})^{\geq n}$. This immediately implies that $F$ has bounded below and coherent cohomology, that is, it belongs to $\mathbf{D}^{+}(\operatorname{noeth}(\mathcal{C}))$. Conversely, for any $F \in \mathbf{D}^{+}(\operatorname{noeth}(\mathcal{C}))$, the truncation triangles easily give us that $F \in \mathbf{K}(\operatorname{Inj} \mathcal{C})^+_c$.

    It trivially follows from the definitions that $\mathbf{K}(\operatorname{Inj} \mathcal{C})^- = \bigcup_{n \in \mathbb{Z}}\mathbf{K}(\operatorname{Inj} \mathcal{C})^{\leq n}$. Finally, from the above computations, it follows that $\mathbf{K}(\operatorname{Inj} \mathcal{C})^b_c = \mathbf{D}^{b}(\operatorname{noeth} \mathcal{C}) \cap \mathbf{K}^{b}(\operatorname{Inj} \mathcal{C})$. 
\end{proof}

\begin{corollary}\label{Coroallary closure of compacts Kinj schemes and algebras}
    Let $X$ be a noetherian, finite dimension scheme such that every integral closed subscheme is J-0. Then, with notation as in \Cref{Convention closure of compacts for approx/co-approx},
    \begin{itemize}
        \item $\mathbf{K}(\operatorname{Inj} X)^+_c = \mathbf{D}^{+}(\operatorname{coh} X)$
        \item $\mathbf{K}(\operatorname{Inj} X)^- = \bigcup_{n \in \mathbb{Z}}\mathbf{K}(\operatorname{Inj} X)^{\leq n}$
        \item $\mathbf{K}(\operatorname{Inj} X)^b_c = \mathbf{D}^{b}(\operatorname{coh} X) \cap \mathbf{K}^{b}(\operatorname{Inj} X)$
    \end{itemize}
    Suppose that $X$ is further J-2, then, for any coherent $\mathcal{O}_X$-algebra $\mathcal{A}$, we have that,
    \begin{itemize}
        \item $\mathbf{K}(\operatorname{Inj} \mathcal{A})^+_c = \mathbf{D}^{+}(\operatorname{coh} \mathcal{A})$
        \item $\mathbf{K}(\operatorname{Inj} \mathcal{A})^- = \bigcup_{n \in \mathbb{Z}}\mathbf{K}(\operatorname{Inj} \mathcal{A})^{\leq n}$
        \item $\mathbf{K}(\operatorname{Inj} \mathcal{A})^b_c = \mathbf{D}^{b}(\operatorname{coh} \mathcal{A}) \cap \mathbf{K}^{b}(\operatorname{Inj} \mathcal{A})$
    \end{itemize}
\end{corollary}
\begin{proof}
    Immediate from \Cref{Proposition closure of compacts for Kinj} using \Cref{Theorem J-2 implies weak co-approximablility for schemes} and \Cref{Theorem J-2 implies weak co-approximablility for algebras}.
\end{proof}

\begin{corollary}\label{Corollary closure of compacts Kmproj schemes}
    Let $X$ be a noetherian, separated, finite dimensional scheme such that every integral closed subscheme is J-0. Theen, with notation as in \Cref{Convention closure of compacts for approx/co-approx},
    \begin{itemize}
        \item $\mathbf{K}_m(\operatorname{Proj} X)^+_c = U_\lambda(\mathbf{D}^{-}(\operatorname{coh} X))^{\circ}$
        \item $\mathbf{K}_m(\operatorname{Proj} X)^b_c = U_\lambda(\mathbf{D}^{\operatorname{perf}}(X))^{\circ}$
    \end{itemize}
    
\end{corollary}
\begin{proof}
    By \Cref{Lemma Coapprox implies preferred equivaence class}, the co-t-structure defined in the proof of \Cref{Proposition Kmproj is Co-approx} lies in the preferred quasiequivalence class, see \Cref{Definition preferred equivalence class}. Hence, the closure of compacts can be computed using that co-t-structure. Then, the fact that $\mathbf{K}_m(\operatorname{Proj} X)^+_c = U_\lambda(\mathbf{D}^{-}(\operatorname{coh} X))^{\circ}$ follows from \cite[Lemma 7.23]{ManaliRahul:2025}.

    Finally, it is clear that $\mathbf{K}_m(\operatorname{Proj} X)^b_c = U_\lambda(\mathsf{S})^{\circ}$, where $\mathsf{S}$ is the full subcategory of $\mathbf{D}^{-}(\operatorname{coh} X)$ defined by,
    \[\mathsf{S} =  \{ F \in \mathbf{D}^{-}(\operatorname{coh} X) : \HomT{F}{\mathbf{D}(\operatorname{coh} X)^{\leq -n}} = 0 \text{ for all } n > > 0 \}\]
    But for any $F \in \mathbf{D}^{-}(\operatorname{coh} X)$, there exist triangles $E_n \to F \to D_n \to \Sigma E_n$ by \cite[Theorem 3.3]{Neeman:2022} and \cite[\href{https://stacks.math.columbia.edu/tag/0FDA}{Tag 0FDA}]{StacksProject}. But, if $F$ is further in $\mathsf{S}$, then the map $F \to D_n$ vanishes for large $n$, which implies $F \in \mathbf{D}^{\operatorname{perf}}(X)$. The other inclusion is standard.
\end{proof}

Now, we come to those cases where we work with the generating sequence $\mathcal{G}$ given by $\mathcal{G}^i = \{\Sigma^{-i} G \}$. We begin with the following result from \cite{Neeman:2022}.

\begin{proposition}\label{Proposition Closure of compacts DQcoh schemes}\cite[page 281]{Neeman:2022} Let $X$ be a quasicompact, quasiseparated scheme, and $Z$ a closed subset with quasicompact complement. Then, the derived category of sheaves with quasicoherent cohomology supported on $Z$, denoted $\mathbf{D}_{\operatorname{Qcoh}}(X)$, has a single compact generator and the full subcategry of compact objects is given by $\mathbf{D}^{\operatorname{perf}}_{Z}(X)$, which is the derived category of perfect complexes supported on $Z$. The closure of compact is given by,
\begin{itemize}
    \item $\mathbf{D}_{\operatorname{Qcoh},Z}(X)^-_c = \mathbf{D}^{p}_{\operatorname{Qcoh},Z}(X)$
    \item $\mathbf{D}_{\operatorname{Qcoh},Z}(X)^b_c = \mathbf{D}^{p,b}_{\operatorname{Qcoh},Z}(X)$
\end{itemize}
    with notation as in \Cref{Convention closure of compacts for approx/co-approx}. 

    $\mathbf{D}^{p}_{\operatorname{Qcoh},Z}(X)$ here denotes the full subcategory of $\mathbf{D}_{\operatorname{Qcoh},Z}(X)$ consisting of the pseudocohrent complexes, while $\mathbf{D}^{p,b}_{\operatorname{Qcoh},Z}(X)$ denotes the full subcategory of pseudocoherent complexes with bounded cohomology. If $X$ is noetherian, then these categories agree with the derived category of sheaves with bounded below and coherent cohomology supported on $Z$, denoted $\mathbf{D}^{-}_{\operatorname{coh},Z}(X)$, and the bounded derived category of sheaves with coherent cohomology supported on $Z$, which is denoted by $\mathbf{D}^{b}_{\operatorname{coh},Z}(X)$, respectively.
\end{proposition}

There is a noncommutative version of the above computation, which we state now.

\begin{proposition}\label{Proposition Closure of compacts DQcoh algebras}
    Let $X$ be a noetherian scheme, and $\mathcal{A}$ a coherent $\mathcal{O}_X$-algebra. Then, $\mathbf{D}_{\operatorname{Qcoh}}(\mathcal{A})$, has a single compact generator and the full subcategory of compact objects is given by $\mathbf{D}^{\operatorname{perf}}(\mathcal{A})$, which is the derived category of perfect complexes. The closure of compact is given by,
\begin{itemize}
    \item $\mathbf{D}_{\operatorname{Qcoh},Z}(\mathcal{A})^-_c = \mathbf{D}^{-}_{\operatorname{coh}}(\mathcal{A})$
    \item $\mathbf{D}_{\operatorname{Qcoh},Z}(\mathcal{A})^b_c = \mathbf{D}^{b}_{\operatorname{coh}}(\mathcal{A})$.
\end{itemize}
    with notation as in \Cref{Convention closure of compacts for approx/co-approx}.
\end{proposition}
\begin{proof}
    This follows easily from \cite[Proposition 4.2]{DeDeyn/Lank/ManaliRahul:2024a}.
\end{proof}

We now consider a stacky example, which follows easily from a couple of results of \cite{Hall/Lamarche/Lank/Peng:2025} and \cite{DeDeyn/Lank/ManaliRahul/Peng:2025}.

\begin{proposition}\label{Proposition Closure of compacts DQcoh stacks}
    Let $\mathcal{X}$ be a noetherian algebraic stack which is concentrated, that is, the canonical morphism $\mathcal{X} \to \operatorname{Spec}(\mathbb{Z})$ is concentrated, see \cite[Definition 2.4]{Hall/Rydh:2017}. Further assume $\mathcal{X}$ satisfies one of the following two conditions,
    \begin{itemize}
        \item $\mathcal{X}$ has quasi-finite and separated diagonal.
        \item $\mathcal{X}$ is a DM stack of characteristic zero.
    \end{itemize}
    Then, $\mathbf{D}_{\operatorname{Qcoh}}(\mathcal{X})$ has a single compact generator and and the closure of compact is given by,
\begin{itemize}
    \item $\mathbf{D}_{\operatorname{Qcoh}}(\mathcal{X})^-_c = \mathbf{D}^{-}_{\operatorname{coh}}(\mathcal{X})$
    \item $\mathbf{D}_{\operatorname{Qcoh}}(\mathcal{X})^b_c = \mathbf{D}^{b}_{\operatorname{coh}}(\mathcal{X})$
\end{itemize}
    with notation as in \Cref{Convention closure of compacts for approx/co-approx}
\end{proposition}

\begin{proof}
    By \cite[Proposition 5.10]{DeDeyn/Lank/ManaliRahul/Peng:2025} the standard t-structure on $\mathbf{D}_{\operatorname{Qcoh}}(\mathcal{X})$ lies in the preferred equivalence class, see \Cref{Definition preferred equivalence class}. Then, the required result follows from \cite{Hall/Lamarche/Lank/Peng:2025} which proves that for any $F \in \mathbf{D}^{-}_{\operatorname{coh}}(\mathcal{X})$ and every integer $n$, there exists a triangle $E_n \to F \to D_n \to \Sigma E_n$ with $E_n \in \mathbf{D}^{\operatorname{perf}}(\mathcal{X})$ and $D_n \in \mathbf{D}_{\operatorname{Qcoh}}(\mathcal{X})^{\leq -n}$.
\end{proof}

\subsection*{Applications}

We now state the applications of the results proved in \Cref{Section Main Results} using the computation of the closure of compacts above done in the previous subsection. In what follows, for any triangulated subcategory $\mathsf{A}\subseteq \mathsf{T}^c$ for a compactly generated triangulated category $\mathsf{T}$, $\operatorname{Coprod}(\mathsf{A})$ would denote the localising subcategory of $\mathsf{T}$ generated by $\mathsf{A}$.

We begin with the following result for quasicompact, quasiseparated schemes. The corresponding statement for noetherian schemes is mentioned in the remark following it.
\begin{theorem}
    Let $X$ be a quasicompact and quasiseparated scheme with a closed subset $\mathsf{Z}$ such that $X\setminus \mathsf{Z}$ is quasicompact. Let $\langle \mathsf{A}, \mathsf{B\rangle}$ be a semiorthogonal decomposition on $\mathbf{D}_{Z}^{\operatorname{perf}}(X)$, which is the derived category of perfect complexes with cohomology supported on $Z$. Then,
    \begin{itemize}
        \item $\langle\operatorname{Coprod}(\mathsf{A}) \cap \mathbf{D}^{p}_{\operatorname{Qcoh},Z}(X),\operatorname{Coprod}(\mathsf{B}) \cap \mathbf{D}^{p}_{\operatorname{Qcoh},Z}(X) \rangle $ is a semiorthogonal decomposition on $\mathbf{D}^{p}_{\operatorname{Qcoh},Z}(X)$, see \Cref{Notation from Neeman}. 
    \end{itemize}
     where $\mathbf{D}^{p}_{\operatorname{Qcoh},Z}(X)$ denotes the derived category of pseudocoherent complexes supported on $Z$.

     If we further assume that $\mathsf{B}$ is an admissible subcategory of $\mathbf{D}_{Z}^{\operatorname{perf}}(X)$, then,
     \begin{itemize}
        \item $\langle\operatorname{Coprod}(\mathsf{A}) \cap \mathbf{D}^{p,b}_{\operatorname{Qcoh},Z}(X),\operatorname{Coprod}(\mathsf{B}) \cap \mathbf{D}^{p,b}_{\operatorname{Qcoh},Z}(X) \rangle $ is a semiorthogonal decomposition on $\mathbf{D}^{p,b}_{\operatorname{Qcoh},Z}(X)$. 
    \end{itemize}
    where $\mathbf{D}^{p,b}_{\operatorname{Qcoh},Z}(X)$ denotes the derived category of pseudocoherent complexes with bounded cohomology supported on $Z$.
\end{theorem}
\begin{proof}
    This is immediate from \Cref{Corollary admissible categories from Tc to T-c} using \Cref{Proposition Closure of compacts DQcoh schemes}.
\end{proof}

We mention what the result says in the noetherian case now, as the categories involved might be more familiar to the reader.

\begin{remark}\label{Remark Main result for Perf of schemes}
    Let $X$ be a noetherian scheme with a closed subset $\mathsf{Z}$. Let $\langle \mathsf{A}, \mathsf{B\rangle}$ be a semiorthogonal decomposition on $\mathbf{D}_{Z}^{\operatorname{perf}}(X)$, which is the derived category of perfect complexes with cohomology supported on $Z$. Then,
    \begin{itemize}
        \item $\langle\operatorname{Coprod}(\mathsf{A}) \cap \mathbf{D}^{-}_{\operatorname{coh},Z}(X),\operatorname{Coprod}(\mathsf{B}) \cap \mathbf{D}^{-}_{\operatorname{coh},Z}(X)\rangle $ is a semiorthogonal decomposition on $\mathbf{D}^{-}_{\operatorname{coh},Z}(X)$, see \Cref{Notation from Neeman}. 
    \end{itemize}
     where $\mathbf{D}^{-}_{\operatorname{coh},Z}(X)$ denotes the derived category of sheaves with bounded above and coherent cohomology supported on $Z$.

     If we further assume that $\mathsf{B}$ is an admissible subcategory of $\mathbf{D}_{Z}^{\operatorname{perf}}(X)$, then,
     \begin{itemize}
        \item $\langle\operatorname{Coprod}(\mathsf{A}) \cap \mathbf{D}^{b}_{\operatorname{coh},Z}(X),\operatorname{Coprod}(\mathsf{B}) \cap \mathbf{D}^{b}_{\operatorname{coh},Z}(X) \rangle $ is a semiorthogonal decomposition on $\mathbf{D}^{b}_{\operatorname{coh},Z}(X)$. 
    \end{itemize}
    where $\mathbf{D}^{b}_{\operatorname{coh},Z}(X)$ denotes the derived category of sheaves with bounded and coherent cohomology supported on $Z$.
\end{remark}

\begin{theorem}\label{Theorem Main result for Dbcoh of schemes}
    Let $X$ be a noetherian finite-dimensional scheme such that each integral closed subscheme is J-0.  Let $\langle \mathsf{A}, \mathsf{B\rangle}$ be a semiorthogonal decomposition on $\mathbf{D}^b({\operatorname{coh}}(X))$. Then,
    \begin{itemize}
        \item  $ \langle \operatorname{Coprod}(\mathsf{A}) \cap \mathbf{D}^{+}(\operatorname{coh} X),\operatorname{Coprod}(\mathsf{B}) \cap \mathbf{D}^{+}(\operatorname{coh} X) \rangle$ is a semiorthogonal decomposition on $\mathbf{D}^{+}(\operatorname{coh} X)$.
    \end{itemize}
    
    If we further assume that $\mathsf{B}$ is an admissible subcategory of $\mathbf{D}^b({\operatorname{coh}}(X))$, then,
     \begin{itemize}
        \item $\langle\mathsf{A} \cap \mathbf{D}^{b}_{\operatorname{coh}}(\operatorname{Inj} X),\mathsf{B} \cap \mathbf{D}^{b}_{\operatorname{coh}}(\operatorname{Inj} X) \rangle $ defines a semiorthogonal decomposition on $\mathbf{D}^{b}_{\operatorname{coh}}(\operatorname{Inj} X)$.
    \end{itemize}
    where $\mathbf{D}^{b}_{\operatorname{coh}}(\operatorname{Inj} X) = \mathbf{D}^{b}_{\operatorname{coh}}(X) \cap \mathbf{K}^{b}(\operatorname{Inj} X)$ denotes the full subcategory of complexes with injective dimension in the bounded derived category of sheaves with coherent cohomology. 
\end{theorem}
\begin{proof}
    This is immediate from \Cref{Corollary admissible categories from Tc to T+c} using \Cref{Proposition closure of compacts for Kinj}.
\end{proof}

\begin{theorem}\label{Theorem Main result for Dbcohop of schemes}
    Let $X$ be a noetherian finite-dimensional scheme such that each integral closed subscheme is J-0.  Let $\langle \mathsf{A}, \mathsf{B\rangle}$ be a semiorthogonal decomposition on $\mathbf{D}^b({\operatorname{coh}}(X))$. Then,
    \begin{itemize}
        \item  $ \langle \operatorname{Coprod}(\mathsf{A}) \cap \mathbf{D}^{-}(\operatorname{coh} X),\operatorname{Coprod}(\mathsf{B}) \cap \mathbf{D}^{-}(\operatorname{coh} X) \rangle$ is a semiorthogonal decomposition on $\mathbf{D}^{-}(\operatorname{coh} X)$.
    \end{itemize}
    
    If we further assume that $\mathsf{B}$ is an admissible subcategory of $\mathbf{D}^b({\operatorname{coh}}(X))$, then,
     \begin{itemize}
        \item $\langle \mathsf{A} \cap \mathbf{D}^{\operatorname{perf}}(X), \mathsf{B} \cap \mathbf{D}^{\operatorname{perf}}(X) \rangle $ is a semiorthogonal decomposition on $\mathbf{D}^{\operatorname{perf}}(X)$.
    \end{itemize}
\end{theorem}
\begin{proof}
    This is immediate from \Cref{Corollary admissible categories from Tc to T+c} using \Cref{Corollary closure of compacts Kmproj schemes}.
\end{proof}

We now state the noncommutative analogues of some of the above results.

\begin{theorem}\label{Theorem Main result for Perf of algebras}
    Let $X$ be a noetherian scheme, and $\mathcal{A}$ a coherent $\mathcal{O}_X$-algebra. Let $\langle \mathsf{A}, \mathsf{B\rangle}$ be a semiorthogonal decomposition on $\mathbf{D}^{\operatorname{perf}}(\mathcal{A})$. Then,
    \begin{itemize}
        \item $\langle\operatorname{Coprod}(\mathsf{A}) \cap \mathbf{D}^{-}_{\operatorname{coh}}(\mathcal{A}),\operatorname{Coprod}(\mathsf{B}) \cap \mathbf{D}^{-}_{\operatorname{coh}}(\mathcal{A})\rangle $ is a semiorthogonal decomposition on $\mathbf{D}^{-}_{\operatorname{coh}}(\mathcal{A})$, see \Cref{Notation from Neeman}. 
    \end{itemize}

     If we further assume that $\mathsf{B}$ is an admissible subcategory of $\mathbf{D}^{\operatorname{perf}}(\mathcal{A})$, then,
     \begin{itemize}
        \item $\langle\operatorname{Coprod}(\mathsf{A}) \cap \mathbf{D}^{b}_{\operatorname{coh}}(\mathcal{A}),\operatorname{Coprod}(\mathsf{B}) \cap \mathbf{D}^{b}_{\operatorname{coh}}(\mathcal{A}) \rangle $ is a semiorthogonal decomposition on $\mathbf{D}^{b}_{\operatorname{coh}}(\mathcal{A})$. 
    \end{itemize}
\end{theorem}
\begin{proof}
    Immediate from \Cref{Corollary admissible categories from Tc to T-c} using \Cref{Proposition Closure of compacts DQcoh algebras}.
\end{proof}

\begin{theorem}\label{Theorem Main result for Dbcoh of algebras}
    Let $X$ be a noetherian finite-dimensional J-2 scheme, and $\mathcal{A}$ a coherent $\mathcal{O}_X$-algebra. Let $\langle \mathsf{A}, \mathsf{B\rangle}$ be a semiorthogonal decomposition on $\mathbf{D}^b({\operatorname{coh}}(\mathcal{A}))$. Then,
    \begin{itemize}
        \item  $ \langle \operatorname{Coprod}(\mathsf{A}) \cap \mathbf{D}^{+}(\operatorname{coh} \mathcal{A}),\operatorname{Coprod}(\mathsf{B}) \cap \mathbf{D}^{+}(\operatorname{coh} \mathcal{A}) \rangle$ is a semiorthogonal decomposition on $\mathbf{D}^{+}(\operatorname{coh} \mathcal{A})$.
    \end{itemize}
    
    If we further assume that $\mathsf{B}$ is an admissible subcategory of $\mathbf{D}^b({\operatorname{coh}}(\mathcal{A}))$, then,
     \begin{itemize}
        \item $\langle\operatorname{Coprod}(\mathsf{A}) \cap \mathbf{D}^{b}_{\operatorname{coh}}(\operatorname{Inj} \mathcal{A}),\operatorname{Coprod}(\mathsf{B}) \cap \mathbf{D}^{b}_{\operatorname{coh}}(\operatorname{Inj} \mathcal{A}) \rangle $ is a semiorthogonal decomposition on $\mathbf{D}^{b}_{\operatorname{coh}}(\operatorname{Inj} \mathcal{A})$.
    \end{itemize}
    where $\mathbf{D}^{b}_{\operatorname{coh}}(\operatorname{Inj} \mathcal{A}) = \mathbf{D}^{b}_{\operatorname{coh}}(\mathcal{A}) \cap \mathbf{K}^{b}(\operatorname{Inj} \mathcal{A})$ denotes the full subcategory of complexes with injective dimension in the bounded derived category of sheaves with coherent cohomology. 
\end{theorem}
\begin{proof}
    This is immediate from \Cref{Corollary admissible categories from Tc to T+c} using \Cref{Proposition closure of compacts for Kinj}.
\end{proof}

Finally, we get the following result for stacks.

\begin{theorem}\label{Theorem Main result for Perf of Stacks}
    Let $\mathcal{X}$ be a noetherian algebraic stack which is concentrated, that is, the canonical morphism $\mathcal{X} \to \operatorname{Spec}(\mathbb{Z})$ is concentrated, see \cite[Definition 2.4]{Hall/Rydh:2017}. Further assume $\mathcal{X}$ satisfies one of the following two conditions,
    \begin{itemize}
        \item $\mathcal{X}$ has quasi-finite and separated diagonal.
        \item $\mathcal{X}$ is a DM stack of characteristic zero.
    \end{itemize}
    
    Let $\langle \mathsf{A}, \mathsf{B\rangle}$ be a semiorthogonal decomposition on $\mathbf{D}^{\operatorname{perf}}(\mathcal{X})$. Then,
    \begin{itemize}
        \item $\langle\operatorname{Coprod}(\mathsf{A}) \cap \mathbf{D}^{-}_{\operatorname{coh}}(\mathcal{X}),\operatorname{Coprod}(\mathsf{B}) \cap \mathbf{D}^{-}_{\operatorname{coh}}(\mathcal{X})\rangle $ is a semiorthogonal decomposition on $\mathbf{D}^{-}_{\operatorname{coh}}(\mathcal{X})$, see \Cref{Notation from Neeman}. 
    \end{itemize}

     If we further assume that $\mathsf{B}$ is an admissible subcategory of $\mathbf{D}^{\operatorname{perf}}(\mathcal{X})$, then,
     \begin{itemize}
        \item $\langle\operatorname{Coprod}(\mathsf{A}) \cap \mathbf{D}^{b}_{\operatorname{coh}}(\mathcal{X}),\operatorname{Coprod}(\mathsf{B}) \cap \mathbf{D}^{b}_{\operatorname{coh}}(\mathcal{X}) \rangle $ is a semiorthogonal decomposition on $\mathbf{D}^{b}_{\operatorname{coh}}(\mathcal{X})$. 
    \end{itemize}
\end{theorem}
\begin{proof}
    This is immediate from \Cref{Corollary admissible categories from Tc to T-c} using \Cref{Proposition Closure of compacts DQcoh stacks}.
\end{proof}

\appendix
\section{Co-approximability for noncommutative coherent algebras}
In this appendix, we will prove some technical results, which are used to prove the co-approximability of the homotopy category of injectives for a coherent algebra over a scheme, see \Cref{Theorem J-2 implies weak co-approximablility for algebras}. We begin with a definition.

\begin{definition}
    A \emph{Noether algebra} is a pair $(X,\mathcal{A})$ where $X$ is a noetherian scheme and $\mathcal{A}$ is a coherent $\mathcal{O}_X$-algebra. We will denote the structure map by $\pi : \mathcal{O}_X \to \mathcal{A}$. The abelian category of quasicoherent (resp.\ coherent) $\mathcal{A}$-modules is denoted by $\operatorname{Qcoh} \mathcal{A}$ (resp.\ $\operatorname{coh} \mathcal{A}$). The full subcategory of injective quasicoherent $\mathcal{A}$-modules is denoted by $\operatorname{Inj} \mathcal{A}$.
\end{definition}

The co-approximability result will follow from \Cref{Proposition weak co-approx for Kinj} using the following theorem. We prove this theorem by a sequence of lemma, analogous to the proof of \Cref{Lemma on approximating on J-2 schemes}.
\begin{theorem}\label{Lemma on approximating on J-2 algebras}
    Let $(X,\mathcal{A})$ be a Noether algebra for a finite dimensional J-2 scheme $X$. Let $\mathsf{T}$ be any triangulated subcategory of $\mathbf{D}(\operatorname{Qcoh} \mathcal{A})$. Then, there is an object $\hat{G}$ in $ \mathbf{D}^{b}(\operatorname{coh} \mathcal{A})$ such that, 
    \begin{enumerate}
        \item There exists an integer $N \geq 0$ such that $\mathsf{T} \cap \operatorname{Qcoh} \mathcal{A} \subseteq \overline{\langle \hat{G} \rangle}^{[-N,N]}$, see \Cref{Notation from Neeman}.
        \item There exist integers $N_{p,q} \geq 0$ for each pair of integers $p \leq q $ such that 
        \[\mathsf{T} \cap \mathbf{D}(\operatorname{Qcoh} \mathcal{A})^{\geq p} \cap \mathbf{D}(\operatorname{Qcoh} \mathcal{A})^{\leq q} \subseteq \overline{\langle \hat{G} \rangle}^{[-N_{p,q},N_{p,q}]}\]
        see \Cref{Notation from Neeman}.
    \end{enumerate}
\end{theorem}
\begin{remark}\label{Remark on approximating on J-2 algebras}
    It is clear that it is enough to show \Cref{Lemma on approximating on J-2 algebras} for $\mathsf{T} = \mathbf{D}(\operatorname{Qcoh} \mathcal{A})$.
    Further, note that \Cref{Lemma on approximating on J-2 algebras}$(1) \implies (2)$. This can be shown by an induction on $(q-p)$ as follows. Suppose $F \in \mathbf{D}(\operatorname{Qcoh} \mathcal{A})^{\geq p} \cap \mathbf{D}(\operatorname{Qcoh} \mathcal{A})^{\leq q}$ for integers $p$ and $q$ such that $q-p \geq 1$. Then, consider the triangle $\tau^{\leq q-1}F \to F \to \tau^{\geq q}F \to \Sigma \tau^{\leq q-1}F$ with $\tau^{\leq q-1}F \in \mathbf{D}(\operatorname{Qcoh} \mathcal{A})^{\geq p} \cap \mathbf{D}(\operatorname{Qcoh} \mathcal{A})^{\leq q-1}$ and $\tau^{\geq q}F \in \mathbf{D}(\operatorname{Qcoh} \mathcal{A})^{\geq q} \cap \mathbf{D}(\operatorname{Qcoh} \mathcal{A})^{\leq q}$ given by the canonical truncation. But, 
    \[\mathbf{D}(\operatorname{Qcoh} \mathcal{A})^{\geq q} \cap \mathbf{D}(\operatorname{Qcoh} \mathcal{A})^{\leq q} = \Sigma^{-q} (\mathbf{D}(\operatorname{Qcoh} \mathcal{A})^{\geq 0} \cap \mathbf{D}(\operatorname{Qcoh} \mathcal{A})^{\leq 0})\]
    so this shows the inductive step.
\end{remark}

\begin{lemma}\label{Lemma 1 on approximating on J-2 algebras}
    Let $(X,\mathcal{A})$ be Noether algebra.
     If $i : Z \to X$ is a closed immersion such that the conclusion of \Cref{Lemma on approximating on J-2 algebras} holds for $(Z,i^\ast(\mathcal{A}))$, then the conclusion of \Cref{Lemma on approximating on J-2 algebras} holds for $(X,\mathcal{A})$ with $\mathsf{T} = \mathbf{D}_{Z}(\operatorname{Qcoh} \mathcal{A})$.
\end{lemma}
\begin{proof}
    Same as proof of \Cref{Lemma 1 on approximating on J-2 schemes}.
\end{proof}

\begin{lemma}\label{Lemma 2 on approximating on J-2 algebras}
    Let $(X,\mathcal{A})$ be a Noether algebra, and let $i_1 : Z_1\to X$ and $i_2 :  Z_2 \to X$ be closed subschemes such that $X = Z_1 \cup Z_2$ topologically. If the conclusion of \Cref{Lemma on approximating on J-2 algebras} holds for $(Z_1,i_1^{\ast}(\mathcal{A}))$ and $(Z_2,i_2^{\ast}(\mathcal{A}))$, then it holds for $(X,\mathcal{A})$.
\end{lemma}
\begin{proof}
    Same as proof of \Cref{Lemma 2 on approximating on J-2 schemes}.
\end{proof}
Now we can prove \Cref{Lemma on approximating on J-2 algebras}.
\begin{proof}[Proof of \Cref{Lemma on approximating on J-2 algebras}]
    Let $(X,\mathcal{A})$ be a Noether algebra, where $X$ is a noetherian finite-dimensional J-2 scheme .
    We start by observing some reductions we can do. First of all, by \Cref{Remark on approximating on J-2 algebras}, it is enough to prove \Cref{Lemma on approximating on J-2 algebras}(1) for $\mathsf{T} = \mathbf{D}(\operatorname{Qcoh} \mathcal{A})$. Further, we can assume $X$ is irreducible using \Cref{Lemma 2 on approximating on J-2 algebras}, and that $X$ is reduced by \Cref{Lemma 1 on approximating on J-2 algebras}. So we can assume that the scheme $X$ is reduced and irreducible, that is, the scheme is integral.

    We now proceed by induction on the dimension of the scheme $X$. For the base case, let $X$ be a integral scheme such that $\dim(X)=0$. Then $X$ is just $\operatorname{Spec}(k)$ for a field $k$ and $\mathcal{A} = \widetilde{A}$ for a finite-dimensional $k$-algebra $A$. Then, by \Cref{Lemma 1 on approximating on J-2 algebras}, we can replace $A$ by $A/J(A)$, where $J(A)$ is the Jacobson radical. As this algebra is semisimple, the result holds.  

    Now, we assume the result is known for dimension smaller than dim$(X)$.
    By \cite[Proposition 4.10]{Elagin/Lunts/Schnurer:2020}, there exists a non-empty affine open set $U = \operatorname{Spec}(R)$, a two sided ideal $I \subseteq \mathcal{A}(U)$ and $\mathcal{O}_{X}(U)$ algebras $A_1, \cdots, A_r$ for some $r \geq 1$, such that each $A_i$ is an Azumaya algebra over its center $Z(A_i)$, each center $Z(A_i)$ is a regular ring, and there is an isomorphism $\mathcal{A}(U)/I \cong A_1 \times \cdots \times A_r$. Let the corresponding open immersion be $j : U \to X$. Let $F \in \operatorname{Qcoh} \mathcal{A}$. Then, the restriction to the open set $U$, $j^*F \in \operatorname{Qcoh} (j^\ast(\mathcal{A})) \cong \operatorname{Mod}(\mathcal{A}(U))$. Now, as $I$ is a nilpotent ideal, there is some $n \geq 1$ such that $I^n = 0$. This gives us a filtration $0 = I^n M \subseteq I^{n-1}M \subseteq \cdots \subseteq IM \subseteq M$. The corresponding triangles allow us to replace the algebra $\mathcal{A}(U)$ by the algebra $A = \mathcal{A}(U)/I$.

    Now, $\operatorname{Mod}(A) = \operatorname{Mod}(A_1) \times \cdots \times \operatorname{Mod}(A_r)$. As $Z(A_i)$ is a regular ring of finite Krull dimension, let $N = \operatorname{max}\{\dim(Z(A_i)) 
    : 1 \leq i \leq r\}$. Then, the global dimension of each $A_i$ is less than or equal to $N$ by \cite[Theorems 1.8 and 2.1]{Auslander/Goldman:1960}. So, $j^*F$ must have a projective resolution of length less than or equal to $N$. This gives us that $j^*F \in \overline{\langle A \rangle}^{[- n, n]}_n = \overline{\langle j^*\mathcal{A} \rangle}^{[- n, n]}$.
    
    Consider the triangle $F' \to F \to \mathbb{R}j_*j^*F \to \Sigma F'$, where the second map is the unit of adjunction. Note that when restricted to $U$, we get that the unit of adjunction is an isomorphism, and so $j^*(F') \cong 0$. Let $Z=X-U$. Then, $F' \in \mathbf{D}_{Z}(\operatorname{Qcoh} \mathcal{A}) $. By \cite[Proposition 3.9.2]{Lipman:2009}, there exists $t \geq 0$  such that $\mathbb{R}j_*j^*F \in \mathbf{D}(\operatorname{Qcoh} \mathcal{A})^{\geq 0} \cap \mathbf{D}(\operatorname{Qcoh} \mathcal{A})^{\leq t-1}$. So, as $F' \in (\Sigma^{-1} \mathbb{R}j_*j^*F) \star F $ we get that, 
    \[F' \in \mathbf{D}_{Z}(\operatorname{Qcoh} \mathcal{A})^{\geq 0} \cap \mathbf{D}_{Z}(\operatorname{Qcoh} \mathcal{A})^{\leq t} \]
    This gives us that,
    \[F \in \big(\mathbf{D}_{Z}(\operatorname{Qcoh} \mathcal{A})^{\geq 0} \cap \mathbf{D}_{Z}(\operatorname{Qcoh} \mathcal{A})^{\leq t} \big)\star \overline{\langle \mathbb{R}j_*j^*\mathcal{\mathcal{A}} \rangle}^{[- n, n]}\]
    as $F \in F' \star \mathbb{R}j_*j^*F$, and as $\mathbb{R}j_*j^*F \in \mathbb{R}j_*\big(\overline{\langle j^*\mathcal{\mathcal{A}} \rangle}^{[- n, n]}\big) \subseteq \overline{\langle \mathbb{R}j_*j^*\mathcal{\mathcal{A}} \rangle}^{[- n, n]}$ by the previous paragraph. Note that the integer $t$ can be chosen independent of the choice of $F$ in $\operatorname{Qcoh} \mathcal{A}$.

    Consider the triangle $Q \to \mathcal{A} \to \mathbb{R}j_*j^*\mathcal{A} \to \Sigma Q$ in $\mathbf{D}(\operatorname{Qcoh} \mathcal{A})$ coming from the unit of adjunction. Note that when restricted to $U$, we get that the unit of adjunction is an isomorphism, and so $j^*(Q) \cong 0$. Therefore, $\Sigma Q \in \mathbf{D}_{Z}(\operatorname{Qcoh} \mathcal{A})^{\geq -1} \cap \mathbf{D}_{Z}(\operatorname{Qcoh} \mathcal{A})^{\leq t}$. Further, $\mathcal{\mathcal{A}} \in \langle G \rangle^{[-C,C]}$ for a compact generator $G \in \mathbf{D}^{\operatorname{perf}}(\mathcal{A})$ and some positive integer $C$. So, $\mathbb{R}j_*j^*\mathcal{\mathcal{A}} \in \langle G \rangle^{[-C,C]} \star \big(\mathbf{D}_{Z}(\operatorname{Qcoh} \mathcal{A})^{\geq -1} \cap \mathbf{D}_{Z}(\operatorname{Qcoh} \mathcal{A})^{\leq t}\big)$. 

    Now, let $i : Z \to X$ be the closed immersion where $Z$ is given the reduced induced closed subscheme structure. First of all, note that as $\dim(Z) < \dim(X)$, $(Z,i^\ast (\mathcal{A}))$ satisfies the conclusion of \Cref{Lemma on approximating on J-2 algebras}. We define $\hat{G} = i_* \hat{H} \oplus G$, where $\hat{H} \in \mathbf{D}^b_{\operatorname{coh}}(i^\ast(\mathcal{A}))$ is the object coming from $(Z,i^\ast (\mathcal{A}))$ satisfying \Cref{Lemma on approximating on J-2 algebras}. As $F \in \big(\mathbf{D}_{Z}(\operatorname{Qcoh} \mathcal{A})^{\geq 0} \cap \mathbf{D}_{Z}(\operatorname{Qcoh} \mathcal{A})^{\leq t} \big)\star \overline{\langle \mathbb{R}j_*j^*\mathcal{A} \rangle}^{[- n, n]}$ and $\mathbb{R}j_*j^*\mathcal{A} \in \overline{\langle \mathcal{G} \rangle}^{[-C,C]} \star \big(\mathbf{D}_{Z}(\operatorname{Qcoh} \mathcal{A})^{\geq -1} \cap \mathbf{D}_{Z}(\operatorname{Qcoh} \mathcal{A})^{\leq t}\big)$ from the above paragraph, we would be done if we can show that for all $p \leq q$, there exists an integer $L_{p,q}$ such that,
    \[\mathbf{D}_{Z}(\operatorname{Qcoh} \mathcal{A})^{\geq p} \cap \mathbf{D}_{Z}(\operatorname{Qcoh} \mathcal{A})^{\leq q} \subseteq \overline{\langle \hat{G} \rangle}^{[-L_{p,q},L_{p,q}]}\]
    But, as the conclusion of \Cref{Lemma on approximating on J-2 algebras} holds for $(Z,i^\ast (\mathcal{A}))$, we get this from \Cref{Lemma 1 on approximating on J-2 algebras}, completing the proof.
\end{proof}

%\Kabeer{Fix the below statement and proof}

\begin{corollary}\label{Theorem on approximating on J-2 algebra}
    Let $(X,\mathcal{A})$ be a Noether algebra with $X$ a finite-dimensional J-2 scheme. Then, there exists an object $\hat{G} \in \mathbf{D}^{b}(\operatorname{coh} \mathcal{A})$ and a positive integer $N$ such that $\mathbf{D}(\operatorname{Qcoh} \mathcal{A})^{\geq 0} = \overline{\langle \hat{G} \rangle}^{[-N,N]} \star \mathbf{D}(\operatorname{Qcoh} \mathcal{A})^{\geq 1}$, see \Cref{Notation from Neeman}.
\end{corollary}
\begin{proof}
    Let $F \in \mathbf{D}(\operatorname{Qcoh} \mathcal{A})^{\geq 0}$. Then, we have the triangle $F^{\leq 0} \to F \to F^{\geq 1} \to \Sigma F^{\leq 0}$ we get from the standard t-structure on $\mathbf{D}(\operatorname{Qcoh} \mathcal{A})$. So, we have that $F^{\leq 0} \in \operatorname{Qcoh} \mathcal{A}$ and $F^{\geq 1} \in \mathbf{D}(\operatorname{Qcoh} \mathcal{A})^{\geq 1}$. But, by \Cref{Lemma on approximating on J-2 algebras}, there exists $N \geq 0$ such that $\operatorname{Qcoh} \mathcal{A} \subseteq \overline{\langle \hat{G} \rangle}^{[-N,N]}$. So, $F \in \overline{\langle \hat{G} \rangle}^{[-N,N]} \star \mathbf{D}(\operatorname{Qcoh} \mathcal{A})^{\geq 1}$, which is what we needed to show.
\end{proof}

\bibliographystyle{alpha}
\bibliography{mainbib}

\end{document}